\documentclass[11.5 pt]{amsart}
\usepackage{amssymb,tikz}
\usepackage{caption}
\usepackage{subcaption}
\usepackage{amsmath}
\usepackage{amsthm}
\usepackage{amsbsy}
\usepackage{psfrag}
\usepackage{pstricks}
\usepackage{graphics}
\usepackage{amssymb} 
\usepackage{graphicx}
\usepackage{parskip}
\usepackage{setspace}
\usepackage[english]{babel}
\usepackage[autostyle]{csquotes}
\usepackage{epstopdf,bm}
\usepackage{enumitem}  
\def\b#1{\boldsymbol{#1}}
\usepackage{enumitem}
\usepackage[autostyle]{csquotes}
\newcounter{example}[section]
\newenvironment{example}[1][]{\refstepcounter{example}\par\medskip
	\textbf{Example~\thesection.\theexample. #1} \rmfamily}{\medskip}
\usepackage{array}
\newcolumntype{P}[1]{>{\centering\arraybackslash}p{#1}}
\theoremstyle{plain}
\newtheorem{theorem}{Theorem}[section]
\newtheorem{lemma}[theorem]{Lemma}

\theoremstyle{remark}

\def\b#1{\boldsymbol{#1}}
\def\norm #1{{\left\vert\kern-0.25ex\left\vert\kern-0.25ex\left\vert #1 
		\right\vert\kern-0.25ex\right\vert\kern-0.25ex\right\vert}}

\begin{document}
	\allowdisplaybreaks[4]
	\numberwithin{figure}{section}
	\numberwithin{table}{section}
	\numberwithin{equation}{section}
	%
	\title[Non-conforming FEM for the quasi-static contact problem]
	{Convergence analysis of Non-Conforming Finite Element Method for quasi-static contact problem}
	\author{Kamana Porwal}
	\email{kamana@maths.iitd.ac.in}
	\address{Department of Mathematics, Indian Institute of Technology Delhi, New Delhi - 110016}
	\author{Tanvi Wadhawan}
	\email{maz188452@maths.iitd.ac.in}
	\address{Department of Mathematics, Indian Institute of Technology Delhi, New Delhi - 110016}
	\date{}
	\begin{abstract}
		In this article, we addressed the numerical solution of a non-linear evolutionary variational inequality, which is encountered in the investigation of quasi-static contact problems. Our study encompasses both the semi-discrete and fully-discrete schemes, where we employ the backward Euler method for time discretization and utilize the lowest order Crouzeix-Raviart nonconforming finite element method for spatial discretization. By assuming appropriate regularity conditions on the solution, we establish \emph{a priori} error analysis for these schemes, achieving the optimal convergence order for linear elements. To illustrate the numerical convergence rates, we provide numerical results on a two-dimensional test problem.
	\end{abstract}
	\maketitle
	\allowdisplaybreaks
	\def\R{\mathbb{R}}
	\def\cA{\mathcal{A}}
	\def\cK{\mathcal{K}}
	\def\cN{\mathcal{N}}
	\def\p{\partial}
	\def\O{\Omega}
	\def\bbP{\mathbb{P}}
	\def\cV{\mathcal{V}}
	\def\cM{\mathcal{M}}
	\def\cT{\mathcal{T}}
	\def\cE{\mathcal{E}}
	\def\cF{\mathcal{F}}
	\def \cW{\mathcal{W}}
	\def \cJ{\mathcal{J}}
	\def \cV{\mathcal{V}}
	\def\bF{\mathbb{F}}
	\def \cW{\mathcal{F}}
	\def\bC{\mathbb{C}}
	\def\bN{\mathbb{N}}
	\def\n{\text{n}}
	\def\ssT{{\scriptscriptstyle T}}
	\def\HT{{H^2(\O,\cT_h)}}
	\def\mean#1{\left\{\hskip -5pt\left\{#1\right\}\hskip -5pt\right\}}
	\def\jump#1{\left[\hskip -3.5pt\left[#1\right]\hskip -3.5pt\right]}
	\def\smean#1{\{\hskip -3pt\{#1\}\hskip -3pt\}}
	\def\sjump#1{[\hskip -1.5pt[#1]\hskip -1.5pt]}
	\def\jumptwo{\jump{\frac{\p^2 u_h}{\p n^2}}}
	
	\section{Introduction}
	\par 
	\noindent
	Contact problems form the core of various mechanical structures and hold significant relevance in hydrostatics and thermostatics. These issues manifest in everyday scenarios, such as the interaction of braking pads with wheels, tires with roads, and pistons with skirts. Due to the significant importance of contact processes in structural and mechanical systems, substantial efforts have been dedicated to their modeling and numerical simulations. The variational inequalities, often serving as variational formulations for these contact problems, have proven to be a potent tool in the mathematical analysis of these models (refer to, for instance, \cite{cocou2018variational, capatina2014existence, sofonea2009variational, shillor2004models, kikuchi1988contact} and related references). In article \cite{kikuchi1988contact}, an early attempt was made to examine frictional contact problems within the framework of variational inequalities. For a comprehensive resource on the analysis and numerical approximations of contact problems involving elastic materials, both with and without friction, one can refer to \cite{duvant2008, han2002quasistatic}. 
	\par 
	In the past decade, research has been extended to study contact problems involving viscoelastic and viscoplastic materials, giving rise to quasi-static contact problems which is a type of time-dependent model. Quasi-static contact problems come into play when the applied forces on a system change gradually over time, resulting in negligible acceleration. We highlight the contributions that have enhanced the mathematical examination of quasi-static frictional contact, including findings from \cite{andersson1991quasistatic, klarbring1988frictional, klarbring1989friction}, which incorporate normal compliance condition, and \cite{cocu1996formulation}, which employs regularization technique. The variational analysis of certain quasi-static contact problems within the context of linearized elasticity is available in references such as \cite{alart1991mixed, andersson2001review,klarbring1988frictional}.  Quasi-static Signorini frictional problems utilizing a nonlocal Coulomb friction law have been examined in \cite{cocu1984existence, cocu1996formulation}. On the other hand, in articles \cite{e2000existence, rocca1999existence, rocca2001existence} the same problem is dealt when a local Coulomb friction law is applied. The \emph{a priori} and \emph{a posteriori} error analysis of finite element methods for linear viscoelastic problems without constraints is detailed in article \cite{fernandez2009priori}. The numerical analysis of a bilateral frictional contact problem for linear elastic material is carried in \cite{barboteu2002numerical}. In articles \cite{djoko2007discontinuous, djoko2007discontinuous1}, discontinuous Galerkin (DG) formulation was devised for the gradient plasticity problem which is formulated as a quasi-static variational inequality of the second type. The article \cite{wang2014discontinuous} focuses on \emph{a priori} error analysis for quasi-static contact problems within a discontinuous Galerkin (DG) framework, considering both semi-discrete and fully-discrete variational formulations. In their work, the authors attained optimal convergence rates when the solution is sufficiently regular.
	\par 
	\noindent
	The main objective of this study is to conduct a numerical analysis of a quasi-static contact problem. In this context, we model the frictional process using Tresca friction law, which includes a prescribed friction bound. This friction bound function represents the maximum magnitude of frictional traction before slip initiates. We investigate both spatially semi-discrete and fully-discrete schemes for this problem, employing finite difference discretization in time and non-conforming finite element discretization in space. 
	\par 
	\noindent
	{\large{\textbf{Mechanical Problem}}}
	\par 
	The quasi-static contact model describes the evolution of the contact process between the linear elastic body and the rigid foundation. Let $\Omega \subset \mathbb{R}^2 $ be the reference configuration of the linear elastic body with boundary $\partial \Omega = \Gamma$. The boundary $\Gamma$ is decomposed into three mutually disjoint, relatively open segments,  the Dirichlet boundary ${\Gamma}_D$, the Neumann boundary ${\Gamma}_N$ and the potential contact boundary ${\Gamma}_C$ with meas ($\Gamma_D)>0$.  The motion of the object is mathematically described using vector-valued functions $\b{u}$ which depends both on position and time. Consequently, a material particle initially situated at the position $\b{x} =(x_1,x_2)$ will find itself at the position $\b{u}(\b{x}, t)$ at time $t$.  Let $\mathcal{I}=[0,T]$ be the time interval of consideration for some $T>0$. The linearized strain tensor, denoted by   $\b{\epsilon}(\b{u}(\b{x}, t))$ belongs to the class of second order symmetric tensors $\mathbb{S}^2$ and is  defined as 
	\begin{align*}
		\b{\epsilon}(\b{u}(\b{x}, t)) := \dfrac{1}{2} (\b{\nabla}(\b{u}(\b{x}, t))+\b{\nabla}(\b{u}(\b{x}, t))^T).
	\end{align*}
	In this context, $t \in I$ represents the time variable and   $\b{u}(\cdot,t) :\Omega  \rightarrow \mathbb{R}^2$ characterizes the displacement variable. The stress tensor $\b{\sigma}(\b{u}(\b{x}, t))$ obeys the following constitutive relation given by Hooke's 
	law 
	\begin{align} \label{sigma}
		\b{\sigma}(\b{u}(\b{x}, t)):= \mathcal{C}\b{u}(\b{x}, t)
	\end{align}
	where  $\mathcal{C} :\mathbb{S}^2 \times \mathbb{R}^2 \rightarrow \mathbb{S}^2$ is the bounded, symmetric and positive definite fourth-order elasticity tensor. In the component form, the relation $\eqref{sigma}$ can be rewritten as 
	\begin{align}\label{paper5_sigma2}
		\sigma_{ij} = \mathcal{C}_{ijkl} \epsilon_{ij} \quad1\leq i,j,k,l \leq 2.
	\end{align}
	Given the homogeneity and isotropy of the linear elastic body, the fourth-order elasticity tensor can be explicitly expressed as follows:
	\begin{align*}
		\mathcal{C}_{ijkl} = \lambda \delta_{ij}\delta_{kl}+\mu (\delta_{ik}\delta_{jl} +\delta_{il}\delta_{jk}), \quad 1\leq i,j,k,l \leq 2,
	\end{align*}
	where, $\lambda$ and $\mu$ are the Lam$\acute{e}$ coefficients and $\delta_{ij}$ is the Kronecker delta. Thus, the stress tensor $\b{\sigma}$ is given by 
	\begin{align*}
		\b{\sigma}((\b{u}(\b{x}, t)))=\lambda(tr(\b{\epsilon}(\b{u}(\b{x}, t)))\textbf{Id} +2\mu{\b{\epsilon}(\b{u}(\b{x}, t))}, 
	\end{align*}
	where $\textbf{Id}$ denotes the identity matrix of order 2. 
	\par 
	\noindent
	First we shall recall the certain functional spaces that shall be relevant in the further analysis. To this end, let $X$ denotes a real Hilbert space with norm $\|\cdot\|_X$.  For $1 \leq p \leq \infty$, let $L^p(\mathcal{I};X)$ denotes the space of measurable functions $v: \mathcal{I} \rightarrow X $ endowed with the norm
	\begin{align*}
		\begin{split}
			\begin{aligned}
				\|v\|_{L^{p}(\mathcal{I};X)} = \begin{cases} &\bigg(\int\limits_{\mathcal{I}}\|v\|^p_X~dt \bigg)^{\frac{1}{p}}~~~\text{if}~ 1\leq p < \infty, \\
					&\underset{t\in \mathcal{I}}{\text{ess~sup}}~\|v\|_X~~~~~~\text{if}~ p=\infty. \end{cases}
			\end{aligned} 
		\end{split}
	\end{align*}
	Further, for $k \in \mathbb{N}$,  the Boncher space $W^{k,p}(\mathcal{I};X)$ is defined as
	\begin{align*}
		W^{k,p}(\mathcal{I};X) : = \bigg\{ {v} \in L^p(\mathcal{I};X): \|v^{(j)}\|_{L^p(\mathcal{I};X)} < \infty,~\forall~j\leq k \bigg\}
	\end{align*}
	and is equipped with the norm
	\begin{align*}
		\begin{split}
			\begin{aligned}
				\|v\|_{	W^{k,p}(\mathcal{I};X)} = \begin{cases} &\bigg(\int\limits_{\mathcal{I}}\sum\limits_{j=1}^k\|v^{(j)}\|^p_X~dt \bigg)^{\frac{1}{p}}~~~~~\text{if}~ 1\leq p < \infty, \\
					&\underset{0\leq j\leq k}{\text{max}
					}~\underset{t\in \mathcal{I}}{\text{ess~sup}}~\|v^{(j)}\|_X~~~~\text{if}~ p=\infty. \end{cases}
			\end{aligned} 
		\end{split}
	\end{align*}
	Here, we used the notation $v^{(j)}$ to denote the distributional derivative of $v$ of order $|j|$. Define $C(\mathcal{I};X)$ to be the space of continuous functions $v : \mathcal{I}\rightarrow X$. 
	\par 
	\noindent
	In this study, we distinguish vector-valued functions using bold symbols, while the scalar-valued functions are denoted in the conventional manner. To be precise, for any Hilbert space $X$, we define $\b{X}:=X\times X.$
	Additionally, we suppress explicit dependence of the quantities on the spatial variable $\b{x},$ i.e. we denote $\b{u}(\b{x},t)$ by $\b{u}(t)$ in the subsequent analysis. Let $\b{n}$ represents the outward unit normal vector to $\Gamma$. Then, for any given vector $\b{v}$, we can decompose it into its normal and tangential component as $v_n:= \b{v}\cdot \b{n}$ and $\b{v}_\tau :=\b{v}-v_n\b{n}$, respectively. In a similar manner, for any tensor-valued function $\b{\Psi}$, we denote $\Psi_n := \b{\Psi}\b{n}\cdot\b{n}$ and $\b{\Psi_{\tau}} := \b{\Psi}\b{n} - \Psi_n\b{n}$ as its normal and tangential component, respectively. Furthermore, we will utilize the subsequent decomposition formula:
	\begin{align}\label{decomp}
		(\b{\Psi}\b{n})\cdot \b{v} = \Psi_nv_n +\b{\Psi_{\tau}}\cdot \b{v}_{\tau}.
	\end{align}
	\par 
	{\underline{\large\textbf{Strong Formulation of Quasi-Static Contact Problem}}}
	\par
	\noindent
	In quasi-static contact problem, we need to find the displacement vector $\bm{u}: \Omega \times [0,T] \rightarrow \mathbb{R}^2$ satisfying the following relations:
	\begin{align}
		\bm{-div} ~~\bm{\sigma}(\bm{u}(t)) &= \bm{f}(t) ~~~~~~~~\textit{in}~\Omega \times (0,T),\label{problem3.3}\\
		\bm{u}(t) &= \bm{0} ~~~~~~~~~~~\textit{on}~\Gamma_D  \times (0,T), \label{problem3.4}\\
		\bm{\sigma}(\bm{u}(t)\bm{n} &= \bm{g}(t) ~~~~~~~~\textit{on}~\Gamma_N  \times (0,T),   \label{problem3.5}\\
		\bm{u}(\bm{x}, 0)&=\bm{u_o} ~~~~~~~~~~\text{in}~~\Omega, \label{problem3.9}  \\
		{u}_{n}= 0, ~|{\bm{\sigma}_\tau}| &\leq~g_a ~~~~~~~~~~{on}~\Gamma_C \times (0,T),  \label{problem3.7}
	\end{align}
	\begin{align}\label{problem3.8}
		\begin{split}
			\left.\begin{aligned}
				|\bm{\bm{\sigma}_\tau}| &< g_a \implies {\bm{\dot{u}_{\tau}}} = \b{0} \\
				|\bm{\bm{\sigma}_\tau}| &= g_a \implies {\bm{\dot{u}_{\tau}}} = -\vartheta \bm{\sigma}_\tau~\textit{for some}~\vartheta \geq 0
			\end{aligned}\right\} ~~~on~\Gamma_C \times (0,T),
		\end{split}
	\end{align}
	where, $\bm{f} \in \bm{W^{1,\infty}}(I; \bm{L^2}(\Omega))$  and $\bm{g} \in \bm{W^{1,\infty}}(I; \bm{L^2}(\Gamma_N))$ represent the density of volume force and surface traction force,  respectively.  Further $g_a \in L^{\infty}(\Gamma_C)$ denotes the frictional bound function. 
	\par 
	The equation \eqref{problem3.3} refers to the equilibrium equation, signifying that the volume force of density $\b{f}(t)$ is acting in $\Omega$. The equation \eqref{problem3.4} justifies that the displacement field vanishes on $\Gamma_D$ indicating that the body is  clamped within this region. A surface traction force with density $\b{g}(t)$ acts on region $\Gamma_N$, resulting in the relationship described in equation \eqref{problem3.5}. The relation \eqref{problem3.9} describes the initial position of the body at time $t=0$. The contact condition described in equation \eqref{problem3.8} is called tresca friction law. Therein, when the strict inequality is satisfied, the material point resides within the \emph{stick zone}.  In contrast, when the equality is met, the material point is in the \emph{slip zone}. The boundary of these zones is not determined {\emph{a priori}} and therefore is an inherent part of the problem.   
	\par
	{\underline{\large\textbf{Variational Formulation of Quasi-Static Contact Problem}}}
	\par 
	In order to establish the variational formulation for the contact problem \eqref{problem3.3}-\eqref{problem3.8},  we introduce Hilbert space $
	\bm{V}$ as
	\begin{align*}
		\bm{V}:= \bigg\{ \bm{v} \in \bm{H^1}(\Omega):   \bm{v} = \bm{0}~ \text{on}~\Gamma_D, v_n=0  ~on~\Gamma_C \bigg\},
	\end{align*}
	The variational formulation for the quasi-static contact problem is to find $\bm{u}: \mathcal{I}\rightarrow \bm{V}$ such that for a.e. $t \in \mathcal{I},$ the following holds
	\begin{align}\label{paper5_1.9}
		a(\bm{u}(t), \bm{v} - \bm{\dot u}(t))+j(\bm{v}) - j(\bm{\dot{u}}(t)) &\geq (\bm{l}(t), \bm{v} -\bm{\dot{u}}(t)) \quad \forall ~\bm{v} \in \bm{V}, \\
		\bm{u}(0) = \bm{u_o}, \label{paper5_1.10}
	\end{align}
	where 
	\begin{align*}
		a(\bm{v, w})&:=\int\limits_\Omega \bm{\sigma} (\bm{v})\colon \bm{\epsilon}(\bm{w})~dx,\\
		j(\bm{v})&:= \int\limits_{\Gamma_C} g_a |\bm{{v_\tau}}|~d\sigma, \\
		(\bm{l}(t), \bm{v})&:= \int\limits_\Omega \bm{f}(t) \cdot \bm{v}~dx + \int\limits_ {\Gamma_N} \bm{g}(t) \cdot \bm{v}~d\sigma~~\forall~ \bm{v},\bm{w}\in\bm{V}.  
	\end{align*}
	Further we assume the initial data satisfy $\bm{u_o} \in \bm{V}$ with 
	\begin{align}\label{VI1.9}
		a(\bm{u_o}, \bm{v})+j(\bm{v})  \geq (\bm{l}(0), \bm{v})~~\forall ~\bm{v} \in \bm{V}.
	\end{align}
	The existence and uniqueness of the solution $\b{u}(t)$ for equations \eqref{paper5_1.9}-\eqref{paper5_1.10} with the assumption \eqref{VI1.9} has been proved in the book \cite{han2002quasistatic}. In addition,  the map $ \psi: \bm{W^{1,\infty}}(I; \bm{V}) \times \bm{V}  \longrightarrow  \bm{L^{\infty}}(I; \bm{V})$ defined by 
	\begin{align*}
		\psi(\bm{l}(t), \bm{u_o}) & \longrightarrow \bm{u}
	\end{align*}
	is a Lipschitz continuous function from $\bm{W^{1,\infty}}(I; \bm{V}) \times \bm{V}  \longrightarrow  \bm{L^{\infty}}(I; \bm{V})$.
	\par 
	The rest of the article is structured as follows: Section 2 begins by introducing the fundamental notations. Following that, we define the discrete space on which the discrete problem is formulated. Therein, we also introduce the Crouzeix-Raviart interpolation operator, denoted by $\mathbf{I}^{\b{CR}}$. This operator exhibits commutativity with the time derivative which is pivotal aspect in the error analysis. The subsequent sections, namely Sections 3 and 4, are devoted to establish \emph{a priori} error estimates for the semi-discrete scheme and the fully-discrete scheme, respectively.
	Finally, in Section 5, we present results of a numerical experiment designed to illustrate the theoretical findings and results discussed in the preceding sections.
	
	\section{Discrete Problem}
	In this segment, we first list the notations which will be useful in the forthcoming analysis.
	\begin{itemize}
		\item  $\mathcal{T}_h$ denotes a regular simplicial triangulation of domain $\O$,
		\item $T$ is an element of $\mathcal{T}_h$,
		\item $\mathcal{V}_T$ refers to the set of all vertices of the triangle $T$,
		\item $\mathcal{V}_h$ refers to the set of all vertices of $\mathcal{T}_h$,
		\item $\mathcal{V}_e$ denotes the set of two vertices of edge $e$,
		\item $\mathcal{E}_h$ denotes the set of all edges of  $\mathcal{T}_h$,
		\item $\mathcal{E}_h^i$ denotes the set of all interior edges of $\mathcal{T}_h$,
		\item $\mathcal{E}_h^b$ denotes the set of all boundary edges of $\mathcal{T}_h$,
		\item $\mathcal{E}_h^D$ denotes the set of all edges lying  on $\Gamma_D$,
		\item $\mathcal{E}_h^C$ denotes the set of all edges lying on $\Gamma_C$,
		\item $\mathcal{E}_h^0$ :=  $\mathcal{E}_h^i \cup  \mathcal{E}_h^D$,
		\item $\mathcal{V}_h^D$ refers to the set of all vertices lying on $\overline{\Gamma}_D$, 
		\item $\mathcal{V}_h^N$ refers to the set of all vertices lying on ${\Gamma}_N$,
		\item $\mathcal{V}_h^C$ refers to the set of all vertices lying on $\overline{\Gamma}_C$,
		\item $\mathcal{M}_h^C$ refers to the set of midpoints of edges lying on ${\Gamma}_C$,
		\item $\mathcal{E}_p$ denotes the set of all edges sharing the node $p$,
		\item $\mathcal{T}_p$ denotes the set of all triangles sharing the node $p$,
				
		\item $\mathcal{T}_e$ denotes the set of all triangles sharing the edge $e$,
		\item  $h_T$ is the diameter of the triangle $T$,
		\item $h$:=  max $\{h_T: T \in \mathcal{T}_h\}$,
		\item  $h_e$ is the length of an edge $e$,
		\item $\mathbb{P}_{k}(T)$ denotes the space of polynomials of degree $\leq k$ defined on $T$ where $0 \leq k \in \mathbb{Z}$,
		\item $|S|$ denotes the cardinality of the set $S$.
	\end{itemize}
	Throughout the analysis, we assume $C$ to be a postive generic constant independent of mesh parameter $h$. Further, we adopt the notation $X \lesssim Y$ to represent $X \leq CY.$
	\par 
	\noindent 
	In order to deal with discontinuous functions conveniently, we introduce the broken Sobolev space $H^1(\Omega, \mathcal{T}_h)$ as follows:
	\begin{align*}
		H^1(\Omega, \mathcal{T}_h) : =\bigg\{ v \in L^2(\Omega) : v|_{T} \in {H}^1(s)~\forall~T\in \mathcal{T}_h\bigg\}.
	\end{align*}
	In the sequel, we will define jumps and averages of scalar, vector and tensor-valued functions. To this end, let $e \in \mathcal{E}^i_h$ be an arbitrary interior edge shared by two adjacent elements $T_1$ and $T_2,$  i.e., $e \in \partial{T}_1 \cap \partial{T}_2$.  Denote by $\b{n}_{T_1}$, the outward unit normal vector of $e$ pointing from $T_1$ to $T_2$ and set $\b{n}_{T_2}=-\b{n}_{T_1}$. The jumps $\sjump{\cdot}$ and averages $\smean{\cdot}$ are defined in the following way:
	\begin{itemize}
		\item For function $w \in 	H^1(\Omega, \mathcal{T}_h)$, define
		\begin{align*}
			\smean{w}&:= \dfrac{w|_{T_1} + w|_{T_2}}{2},~~~\sjump{w}:= w|_{T_1}\b{n}_{T_1} + w|_{T_2}\b{n}_{T_2} ; 
		\end{align*} 
		\item For any function $\b{v} \in [H^1(\Omega, \mathcal{T}_h)]^2$, define
		\begin{align*}
			\smean{\b{v}}&:= \dfrac{\b{v}|_{T_1} + \b{v}|_{T_2}}{2},~~~~\sjump{\b{v}}:= \b{v}|_{T_1}\otimes\b{n}_{T_1} + \b{v}|_{T_2}\otimes\b{n}_{T_2};
		\end{align*} 
		\item For function $\b{\phi} \in [H^1(\Omega, \mathcal{T}_h)]^{2 \times 2}$, set
		\begin{align*}
			\smean{\b{\phi}}&:= \dfrac{\b{\phi}|_{T_1} + \b{\phi}|_{T_2}}{2},~~~\sjump{\b{\phi}}:= \b{\phi}|_{T_1}\b{n}_{T_1} + \b{\phi}|_{T_2}\b{n}_{T_2};
		\end{align*}
	\end{itemize}
	Here, the notation $a\otimes b$ is used to describe the dyadic product of two vectors $a, b \in \mathbb{R}^2.$  
	Analogously, for the notational convenience, we define jumps and averages for the scalar-valued function $w \in {H^1}(\O, \mathcal{T}_h) $,  vector-valued function $\b{v} \in \b{H^1}(\O, \mathcal{T}_h)$ and tensor-valued function $\b{\phi} \in [{H^1}(\O, \mathcal{T}_h)]^{2 \times 2} $ on the boundary edges $e \in \mathcal{E}^b_h$ as follows
	\begin{align*}
		\smean{w}&:=w,~~~~~~~~~\sjump{w}:= w\b{n_e} ; \\
		\smean{\b{v}}&:= \b{v},~~~~~~~~~~\sjump{\b{v}}:= \b{v}\otimes\b{n_e} ; \\
		\smean{\b{\phi}}&:=  \b{\phi},~~~~~~~~~\sjump{\b{\phi}}:= \b{\phi}\b{n_e} ;
	\end{align*}
	where $T \in \mathcal{T}_h$ contains the edge $e$ and $\b{n_e}$ is the outward unit normal on the edge $e$ pointing outside $T$.
	\par 
	\noindent
	Next, we introduce Crouziex-Raviart finite element space to discretize problem \eqref{paper5_1.9}-\eqref{paper5_1.10} as
	\begin{align}
		\b{V^h_{C R}}:= &\bigg\{\b{v_{h}} \in \b{L^{2}}(\Omega):\b{v_{h}}|_{T} \in \mathbb{P}_{1}(T)~\forall ~T \in \mathcal{T}_{h}, \int_{e}\sjump{\b{v_{h}}}~ \mathrm{d}\sigma=0 \quad \forall~e \in \mathcal{E}_{h}^{i},\notag\\
		&\text { and } \b{v_{h}}\left(m_{e}\right)=0 \quad \forall~ m_e \in \mathcal{M}_{h}^{D}\bigg\},
	\end{align}
	and
	\begin{align}
		\b{\mathcal{K}^h_{C R}}:=\left\{\b{v_{h}} \in \b{V^h_{C R}}\left(\mathcal{T}_{h}\right): v_{hn}\left(m_{e}\right) = 0 \quad \forall~m_e \in \mathcal{M}_{h}^{C}\right\}~~~~~~~~~~~~~~~
	\end{align}
	where $\b{n}$ denotes the outward unit normal vector on $\Gamma_C$.
	\par 
	\noindent
	At the several occasions,  we will be using the following inverse and trace inequalities.
	\begin{itemize}
		\item {\large{\textbf{Trace inequality}}} \cite{scott1990finite, ciarlet2002finite}
		\begin{lemma}\label{discrete_trace}
			Let $T \in \cT_h$ and $e$ be an edge of $T$.  Then for ${\phi} \in L^p(s),~1\leq p<\infty$,  the following holds:
			\begin{eqnarray*}
				\|\phi\|^p_{L^p(e)} \lesssim h_e^{-1} \Big( \|\phi\|^p_{L^p(T)} + h_e^p \|\nabla \phi\|^p_{L^p(T)}\Big).
			\end{eqnarray*}
		\end{lemma}
		\vspace{0.3 cm}
		\item  {\large{\textbf{Inverse inequalities}}}~\cite{scott1990finite, ciarlet2002finite}
		\begin{lemma} \label{inverse_inequality}
			Let $1 \leq s,t \leq \infty$ and $\b{{v}_h} \in  \b{V^h_{C R}}$. Then, it holds that
			\begin{enumerate}
				\item $	\|\b{{v}_h}\|_{{W}^{m,s}(T)} \lesssim h^{l-m}_T h_T^{2(\frac{1}{s}-\frac{1}{t})} \|\b{{v}_h}\|_{{W}^{l,t}(T)} \quad \forall~T\in \cT_h$, ~$l \leq m$,
				\item $	\|\b{{v}_h}\|_{{L}^{\infty}(T)} \lesssim h^{-1}_T \|\b{{v}_h}\|_{{L}^2(T)} \quad \forall~ T \in \cT_h$,
				\item $	\|\b{{v}_h}\|_{{L}^{\infty}(e)} \lesssim h^{-\frac{1}{2}}_e \|\b{{v}_h}\|_{{L}^2(e)} \quad \forall~e \in \cE_h$.
			\end{enumerate}	
		\end{lemma}
		\vspace{0.3 cm}
	\end{itemize}
	\par 
	\noindent 
		Next, we will define the following interpolation operator which will play a crucial role in the upcoming analysis. Define the interpolation operator $\mathbf{I}^{\b{CR}}: \b{H^1}(\Omega) \longrightarrow \b{V^h_{CR}}$ by 
	\begin{align*}
		\mathbf{I}^{\b{CR}}(\b{v}) : =\sum_{e\in \mathcal{E}_h}\bigg(\int_e \frac{1}{h_e}\b{v}~d\sigma\bigg) \psi_e,
	\end{align*}
	where $\psi_e$ refers to the edge oriented basis function for the space $\b{V^h_{CR}}$ which satisfies the following properties:
	\par 
	\begin{enumerate}
		\item $\psi_e$ should be equal to one along edge $e$.
		\item $\psi_e(m_{e^{'}})=0 $ for any other edge $e'\in \mathcal{E}^h~ \backslash~e.$  
		\item Support of $\psi_e$ is contained in $\mathcal{T}_e.$
	\end{enumerate}
	The crucial property of this map is that it commutes with the time derivative, that is 
	\begin{align*}
		\frac{\partial}{\partial t}\mathbf{I}^{\b{CR}}(\b{v}) &= \frac{\partial}{\partial t}\sum_{e\in \mathcal{E}_h}\bigg( \frac{1}{h_e}\int_e\b{v}~d{\sigma}\bigg) \psi_e \\
		&= \sum_{e\in \mathcal{E}_h}\bigg( \frac{1}{h_e}\int_e\frac{\partial \b{v}}{\partial t}~d{\sigma}\bigg) \psi_e \\
		&= \mathbf{I}^{\b{CR}}\bigg(\frac{\partial \b{v}}{\partial t}\bigg).
	\end{align*}
	Next, we will see the approximation properties of the interpolation map $\mathbf{I}^{\b{CR}}$ in the next lemma \cite{carstensen2017nonconforming, ciarlet2002finite}.
	\begin{lemma}\label{Interpolation1}
		For any $\b{v} \in \b{H^\nu}(T),  T \in \mathcal{T}_h$,  the following holds
		\begin{align*}
			|\b{v}-\mathbf{I}^{\b{CR}}(\b{v})|_{\b{H^l}(T)} \lesssim  h_T^{\nu-l}|\b{v}|_{\b{H^{\nu}}(T)}~~~~\text{for}~~0\leq l\leq \nu,
		\end{align*}
		where $~1\leq\nu\leq2$.
	\end{lemma}
	Note that for the sake of notational convenience, we denote $\mathbf{I}^{\b{CR}}(\b{u}(t))$ by $\b{u_I}(t)$ in the upcoming sections. 
	\par 
	\noindent
	In this analysis, we introduce and examine two types of approximation schemes. When solely the spatial variables are discretized, we obtain semi-discrete schemes. On the other hand, when both the spatial and temporal variables are discretized, we arrive at fully-discrete schemes. For both the schemes, we aim to provide rigorous {\emph{a priori}} error analysis and achieve optimal order of convergence under additional assumption on solution's regularity.
	\section{Semi-Discrete Approximation}\label{ch5sec4}
	The spatially semi-discrete formulation for the variational inequality \eqref{paper5_1.9}-\eqref{paper5_1.10} is given by
	
	\textbf{Problem 1}~(\textit{Semi-Discrete Variational Inequality})\textit{ Find $\b{u_h}:\mathcal{I} \rightarrow \b{\mathcal{K}^h_{C R}}$ satisfying
		\begin{align}\label{paper5_1.99}
			a_h(\bm{u_h}(t), \bm{v_h} - \bm{\dot{u}_h}(t))+j(\bm{v_h}) - j(\bm{\dot{u}_h}(t)) &\geq (\bm{l}(t), \bm{v_h} -\bm{\dot{u}_h}(t)) \quad \forall ~\bm{v_h} \in \b{\mathcal{K}^h_{C R}},\\
			\bm{u_h}(0) =\mathbf{I}^{\b{CR}}\bm{u_o}. \label{paper5_1.110}
	\end{align}}
	Here,  the discrete bilinear form $a_h: \b{V^h_{C R}} \times \b{V^h_{C R}} \rightarrow \mathbb{R}$ is defined as 
	\begin{align*}
		a_h(\bm{v_h, w_h})&:=\sum_{T \in \mathcal{T}_h}\int\limits_T \bm{\sigma} (\bm{v_h})\colon \bm{\epsilon}(\bm{w_h})~dx + \sum_{e \in \mathcal{E}^0_h}2\rho\mu\int\limits_e \frac{1}{h_e} \sjump{\b{v_h}}:\sjump{\b{w_h}}~d{\sigma}.
	\end{align*}
	Note that the parameter $\rho >0$ is introduced in the discrete bilnear form $a_h(\cdot, \cdot)$ in order to stabalize it \cite{hansbo2003discontinuous}. Next, we define the norm $\b{\norm{ \cdot }}$ on the space $\b{V^h_{CR}}+\b{V}\cap [H^2(\Omega)]^2$ using the following relations
	\begin{align*}
		\norm{\b{v}}^2:=\mid \b{v}\mid^2_h+\mid \b{v}\mid_*^2,
	\end{align*}
	where
	\begin{align*}
		\mid \b{v} \mid^2_h := \sum_{T\in\mathcal{T}_h} \mid \b{v}\mid_T^2,
		~~~~~~~~~~~~
		\mid \b{v} \mid^2_* := \sum_{e\in\mathcal{E}_h^o}\frac{2\rho\mu}{h_e}\| \sjump{\b{v}}\|^2_{\b{L^2(e)}},
	\end{align*}
	with
	\begin{align*}
		\mid \b{v} \mid^2_T := \int\limits_{T} \b{\sigma}(\b{v}):\b{\epsilon}(\b{v})~dx,~~~~~~~
		\| \sjump{\b{v}}\|^2_{\b{L^2(e)}}:= \int\limits_e\sjump{\b{v}}:\sjump{\b{v}}~d{\sigma}.
	\end{align*}
	Note that the bilinear form $a_h(\cdot,\cdot)$ is continuous and coercive with respect to the norm $\norm{\cdot}$. Thus, the unique solvability of the Problem 1 follows from Stampacchia theorem \cite{ciarlet2002finite}. 
	\subsection {{{Error Analysis for Spatially Semi-Discrete Approximation}}}
	\par 
	\noindent
	\vspace{0.3 cm}
	\\ 
	In the next segment, we aim to derive ${\emph{a priori}}$ error estimates for the semi-discrete variational inequality \eqref{paper5_1.99}-\eqref{paper5_1.110} taking into the account appropriate regularity assumptions on the exact solution $\b{u}(t)$. The frictional contact condition results in a singular behavior of $\b{u(t)}$ near the free boundary around $\Gamma_C$, even when the forces $\b{f}(t)$ and $\b{g}(t)$ exhibit sufficient regularity. Keeping this in mind, it is realistic to assume that $\b{u}(t)$ and $\b{\dot}{\b{u}}(t) $ belongs to $\b{L^{\infty}}[\mathcal{I};\b{H^2}(\Omega)].$ The following theorem guarantees optimal order {\emph{a priori}} error estimates for the quasi-static contact problem. Therein, the error is measured in $\b{L^{\infty}}(\mathcal{I};\b{V_h})$ norm and it is shown to deacy with order $\mathcal{O}(h).$
	\begin{theorem}
		Let $\b{u}(t)$ be the solution of continuous problem \eqref{paper5_1.9}-\eqref{paper5_1.10} and $\b{u_h}(t)$ be the solution of semi-discrete problem \eqref{paper5_1.99}-\eqref{paper5_1.110}. Assume $\b{u}(t),\b{\dot}{\b{u}}(t) \in \b{L^{\infty}}[\mathcal{I};\b{H^2}(\Omega)]$ for any $t \in \mathcal{I}$. Then, we have 
		\begin{align*}
			\norm{\b{u}(t) -\b{u_h}(t)}_{\b{L^{\infty}}(\mathcal{I};\b{V_h})} \lesssim h
		\end{align*}
		for any $t \in [0,T].$
		\vspace{-0.2 cm}
		\begin{proof}
			Let $\b{e}(t)=\b{u}(t) - \b{u_h}(t) = \b{\theta}(t) +\b{\eta}(t)$ where $\b{\theta}(t) := \b{u}(t) -\b{u_I}(t)$ and $\b{\eta}(t) :=\b{u_I}(t) - \b{u_h}(t)$.
			Now, employing the triangle inequality we obtain
			\begin{align*}
				\norm{\b{e}(t)} \leq \norm{\b{\theta}(t)}+\norm{\b{\eta}(t)}.
			\end{align*} 
			First we will derive an upper bound for $\norm{\b{\theta}(t)}$ using Lemma \ref{Interpolation1}  as follows
			\begin{align*}
				\norm{\b{\theta}(t)}^2 = \sum_{T \in \mathcal{T}_h} |\b{u}(t)-\b{u_I}(t)|^2_{\b{H^1}(T)} \leq  \sum_{T \in \mathcal{T}_h} h^2_T |\b{u}(t)|^2_{\b{H^2}(T)}.
			\end{align*}
			Next we need to bound $\norm{\b{u_I}(t) - \b{u_h}(t)}$. To this end, we have
			\begin{align}\label{paper5_2}
				\frac{1}{2}\frac{\partial}{\partial{t}}\norm{\b{u_I}(t) - \b{u_h}(t)}^2 &= 	\frac{1}{2}\frac{\partial}{\partial{t}} \bigg( \sum_{T \in \mathcal{T}_h} \int\limits_T \b{\sigma}(\b{u_I}(t) -\b{u_h}(t)) : \b{\epsilon}(\b{u_I}(t) -\b{u_h}(t))~dx \nonumber\\ &+ \sum_{e\in \mathcal{E}^0_h} \frac{2\rho\mu}{h_e} \int\limits_e \sjump{\b{u_I}(t)-\b{u_h}(t)}:\sjump{\b{u_I}(t)-\b{u_h}(t)}~d{\sigma} \bigg) \nonumber\\
				&= a_h(\b{u_I}(t)-\b{u_h}(t), \b{\dot}{\b{u_I}}(t) -\b{\dot}{\b{u_h}}(t)) \notag\\
				& =a_h(\b{u_I}(t)-\b{u}(t), \b{\dot}{\b{u_I}}(t) -\b{\dot}{\b{u_h}}(t)) + a_h(\b{u}(t), \b{\dot}{\b{u_I}}(t) -\b{\dot}{\b{u_h}}(t)) \nonumber \\&- a_h(\b{u_h}(t), \b{\dot}{\b{u_I}}(t) -\b{\dot}{\b{u_h}}(t)). 
			\end{align}
			Using the fact that $ \b{\dot}{\b{u_I}}(t) \in \b{\mathcal{K}^h_{CR}}$, we find 
			\begin{align*}
				a_h(\bm{u_h}(t),\b{\dot}{\b{u_I}}(t)- \bm{\dot u_h}(t))+j( \b{\dot}{\b{u_I}}(t)) - j(\bm{\dot{u_h}}(t)) &\geq (\bm{l}(t),  \b{\dot}{\b{u_I}}(t) -\bm{\dot{u_h}}(t)).
			\end{align*}
			Thus, 
			\begin{align}\label{paper5_1}
				-a_h(\bm{u_h}(t),\b{\dot}{\b{u_I}}(t)- \bm{\dot u_h}(t)) \leq j( \b{\dot}{\b{u_I}}(t)) - j(\bm{\dot{u_h}}(t)) - (\bm{l}(t),  \b{\dot}{\b{u_I}}(t) -\bm{\dot{u_h}}(t)).
			\end{align}
			Inserting equation $\eqref{paper5_1}$  in the relation $\eqref{paper5_2}$, we obtain
			\begin{align}\label{paper5_4}
				\frac{1}{2}\frac{\partial}{\partial{t}}\norm{\b{u_I}(t) - \b{u_h}(t)}^2 &\leq a_h(\b{u_I}(t)-\b{u}(t), \b{\dot}{\b{u_I}}(t) -\b{\dot}{\b{u_h}}(t)) + a_h(\b{u}(t), \b{\dot}{\b{u_I}}(t) -\b{\dot}{\b{u_h}}(t))  +  j( \b{\dot}{\b{u_I}}(t)) \nonumber \\& - j(\bm{\dot{u_h}}(t))- (\bm{l}(t),  \b{\dot}{\b{u_I}}(t) -\bm{\dot{u_h}}(t)).
			\end{align}
			Applying integration by parts to the second term of the last relation, we obtain
			\begin{align}\label{paper5_3}
				a_h(\b{u}(t), \b{\dot}{\b{u_I}}(t) -\b{\dot}{\b{u_h}}(t)) &=  \sum_{T \in \mathcal{T}_h} \int\limits_T \b{\sigma}(\b{u}(t)) : \b{\epsilon}(\b{\dot}{\b{u_I}}(t) -\b{\dot}{\b{u_h}}(t))~dx \nonumber\\
				&= \sum_{T\in \mathcal{T}_h} \int\limits_{T} \b{f}(t)\cdot (\b{\dot}{\b{u_I}}(t) -\b{\dot}{\b{u_h}}(t))~dx +\sum_{e\in \mathcal{E}^N_h} \int\limits_{e} \b{g}(t)\cdot (\b{\dot}{\b{u_I}}(t) -\b{\dot}{\b{u_h}}(t))~d{\sigma} \nonumber\\&+ \sum_{e\in \mathcal{E}^C_h} \int\limits_{e} (\b{\sigma}(\b{u}(t))\b{n}) \cdot (\b{\dot}{\b{u_I}}(t) -\b{\dot}{\b{u_h}}(t))~d{\sigma} \nonumber\\&+ \sum_{e\in \mathcal{E}^0_h} \int\limits_{e} \smean{\b{\sigma}(\b{u}(t))}: \sjump{\b{\dot}{\b{u_I}}(t) -\b{\dot}{\b{u_h}}(t)}~d{\sigma}.
			\end{align}	
			Clubbing equation \eqref{paper5_3} in equation \eqref{paper5_4}, we find
			\begin{align*}
				\frac{1}{2}\frac{\partial}{\partial t}\norm{\b{u_I}(t) - \b{u_h}(t)}^2 &\leq a_h(\b{u_I}(t)-\b{u}(t), \b{\dot}{\b{u_I}}(t) -\b{\dot}{\b{u_h}}(t)) \nonumber \\&+  j( \b{\dot}{\b{u_I}}(t)) - j(\bm{\dot{u_h}}(t)) + \sum_{e\in \mathcal{E}^C_h} \int\limits_{e} (\b{\sigma}(\b{u}(t))\b{n}) \cdot (\b{\dot}{\b{u_I}}(t) -\b{\dot}{\b{u_h}}(t))~d{\sigma} \nonumber \\&+ \sum_{e\in \mathcal{E}^0_h} \int\limits_{e} \smean{\b{\sigma}(\b{u}(t))}: \sjump{\b{\dot}{\b{u_I}}(t) -\b{\dot}{\b{u_h}}(t)}~d{\sigma}.
			\end{align*}
			Integrating the last inequality from $0$ to $t$, we obtain 
			\begin{align*}
				\norm{\b{u_I}(t) - \b{u_h}(t)}^2 &\leq \norm{\b{u_I}(0) - \b{u_h}(0)}^2 + 2\int\limits_{0}^{t} a_h(\b{u_I}(s)-\b{u}(s), \b{\dot}{\b{u_I}}(s) -\b{\dot}{\b{u_h}}(s))~ds\\& +2\int\limits_{0}^{t} \big(j( \b{\dot}{\b{u_I}}(s)) - j(\bm{\dot{u_h}}(s)) \big)~ds + 2\int\limits_{0}^{t}  \sum_{e\in \mathcal{E}^C_h} \int\limits_{e} (\b{\sigma}(\b{u}(s))\b{n}) \cdot (\b{\dot}{\b{u_I}}(s) -\b{\dot}{\b{u_h}}(s))~d\sigma ds\\&+ 2\int\limits_{0}^{t} \sum_{e\in \mathcal{E}^0_h} \int\limits_{e} \smean{\b{\sigma}(\b{u}(s))}: \sjump{\b{\dot}{\b{u_I}}(s) -\b{\dot}{\b{u_h}}(s)}~d\sigma ds.
			\end{align*}
			Exploiting the fact that $\b{u_I}(0) = \b{u_h}(0)$ and $\sigma_n(\b{u}(t)) =0$ on $\Gamma_C$, we have
			\begin{align}\label{paper5_5}
				\norm{\b{u_I}(t) - \b{u_h}(t)}^2 &\leq 2\int\limits_{0}^{t} a_h(\b{u_I}(s)-\b{u}(s), \b{\dot}{\b{u_I}}(s) -\b{\dot}{\b{u_h}}(s))~ds  \nonumber \\&+ 2\int\limits_{0}^{t}  \sum_{e\in \mathcal{E}^C_h} \int\limits_{e} \b{\sigma}_{\tau}(\b{u}(s)) \cdot (\b{\dot}{\b{u_I}}(s) -\b{\dot}{\b{u_h}}(s))_{\tau}~d{\sigma} ds \nonumber +2\int\limits_{0}^{t}\big( j( \b{\dot}{\b{u_I}}(s)) - j(\bm{\dot{u_h}}(s))\big)~ds  \nonumber \\&+ 2\int\limits_{0}^{t} \sum_{e\in \mathcal{E}^0_h} \int\limits_{e} \smean{\b{\sigma}(\b{u}(s))}: \sjump{\b{\dot}{\b{u_I}}(s) -\b{\dot}{\b{u_h}}(s)}~d\sigma ds  \nonumber\\
				&= A_1 +A_2+A_3,
			\end{align}
			where
			\begin{align*}
				A_1 &= 2\int\limits_{0}^{t} a_h(\b{u_I}(s)-\b{u}(s), \b{\dot}{\b{u_I}}(s) -\b{\dot}{\b{u_h}}(s))~ds,\\
				A_2 &= 2\int\limits_{0}^{t}  \sum_{e\in \mathcal{E}^C_h} \int\limits_{e} \b{\sigma}_{\tau}(\b{u}(s)) \cdot (\b{\dot}{\b{u_I}}(s) -\b{\dot}{\b{u_h}}(s))_{\tau}~d{\sigma} ds +2\int\limits_{0}^{t}\big( j( \b{\dot}{\b{u_I}}(s)) - j(\bm{\dot{u_h}}(s))\big)~ds,  \\
				A_3 &=  2\int\limits_{0}^{t} \sum_{e\in \mathcal{E}^0_h} \int\limits_{e} \smean{\b{\sigma}(\b{u}(s))}: \sjump{\b{\dot}{\b{u_I}}(s) -\b{\dot}{\b{u_h}}(s)}~d\sigma ds.
			\end{align*}
			We will now address each term individually. To establish an upper bound for $A_1$, we conduct integration by parts with respect to time as follows:
			\begin{align*}
				A_1 &=  2\int\limits_{0}^{t} a_h(\b{u_I}(s)-\b{u}(s), \b{\dot}{\b{u_I}}(s) -\b{\dot}{\b{u_h}}(s))~ds  \nonumber \\
				&= 2  \sum_{T \in \mathcal{T}_h} \int\limits_T \b{\sigma}(\b{u_I}(t)-\b{u}(t)) : \b{\epsilon}({\b{u_I}}(t) -{\b{u_h}}(t))~dx \\&- 2 \sum_{T \in \mathcal{T}_h} \int\limits_T \b{\sigma}(\b{u_I}(0)-\b{u}(0)) : \b{\epsilon}({\b{u_I}}(0) -{\b{u_h}}(0))~dx\\
				&-2\int\limits_{0}^{t}  \sum_{T \in \mathcal{T}_h} \int\limits_T \b{\sigma}(\b{\dot}{\b{u_I}}(s)-\b{\dot}{\b{u}}(s)) : \b{\epsilon}({\b{u_I}}(s) -{\b{u_h}}(s))~dxds \\
				&\lesssim \sum_{T \in \mathcal{T}_h}|\b{u_I}(t) -\b{u}(t)|_{\b{H^1}(T)}|\b{u_I}(t) -\b{u_h}(t)|_{\b{H^1}(T)} +  \sum_{T \in \mathcal{T}_h}|\b{u_I}(0) -\b{u}(0)|_{\b{H^1}(T)}|\b{u_I}(0) -\b{u_h}(0)|_{\b{H^1}(T)} \\
				& + \int\limits_{0}^{t} \sum_{T \in \mathcal{T}_h}|\b{\dot}{\b{u_I}}(s)-\b{\dot}{\b{u}}(s)|_{\b{H^1}(T)}|\b{u_I}(s) -{\b{u_h}}(s)|_{\b{H^1}(T)}~ds.
			\end{align*}
			A use of Young's inequality in the last relation yields
			\begin{align*}
				A_1 
				&\leq \dfrac{1}{\epsilon}\sum_{T \in \mathcal{T}_h}|\b{u_I}(t) -\b{u}(t)|^2_{\b{H^1}(T)} + \epsilon \sum_{T \in \mathcal{T}_h}|\b{u_I}(t) -\b{u_h}(t)|^2_{\b{H^1}(T)} + \dfrac{1}{\epsilon} \sum_{T \in \mathcal{T}_h}|\b{u_I}(0) -\b{u}(0)|^2_{\b{H^1}(T)}  \nonumber \\&+ \epsilon\sum_{T \in \mathcal{T}_h}|\b{u_I}(0) -\b{u_h}(0)|^2_{\b{H^1}(T)} 
				+ \int\limits_{0}^{t} \dfrac{1}{\epsilon} \sum_{T \in \mathcal{T}_h}|\b{\dot}{\b{u_I}}(s)-\b{\dot}{\b{u}}(s)|^2_{\b{H^1}(T)}~ds \\&+  \int\limits_{0}^{t}{\epsilon} \sum_{T \in \mathcal{T}_h} |\b{u_I}(s) -{\b{u_h}}(s)|^2_{\b{H^1}(T)}~ds
			\end{align*}
			for some $\epsilon>0$. Using Lemma \ref{Interpolation1}, we obtain
			\begin{align*}
				A_1 
				&\leq \dfrac{1}{\epsilon}\sum_{T \in \mathcal{T}_h} h_T^2 |\b{u}(t)|^2_{\b{H^2}(T)}+ \epsilon \sum_{T \in \mathcal{T}_h}|\b{u_I}(t) -\b{u_h}(t)|^2_{\b{H^1}(T)} + \dfrac{1}{\epsilon} \sum_{T \in \mathcal{T}_h}|\b{u_I}(0) -\b{u}(0)|^2_{\b{H^1}(T)} \\&+ \epsilon\sum_{T \in \mathcal{T}_h}|\b{u_I}(0) -\b{u_h}(0)|^2_{\b{H^1}(T)} 
				+ \int\limits_{0}^{t} \dfrac{1}{\epsilon} \sum_{T \in \mathcal{T}_h}h_T^2|\b{\dot}{\b{u}}(s)|^2_{\b{H^2}(T)}~ds+  \int\limits_{0}^{t}{\epsilon}\norm{\b{u_I}(s) -{\b{u_h}}(s)}^2~ds \\
				&\leq \dfrac{1}{\epsilon} h^2 \|\b{u}\|^2_{\b{L^{\infty}}(\mathcal{I}; \b{H^2}(\O))}+ \epsilon \norm{\b{u_I} -\b{u_h}}^2_{\b{L^{\infty}}(\mathcal{I}; \b{V_h})} + \dfrac{1}{\epsilon} \sum_{T \in \mathcal{T}_h}|\b{u_I}(0) -\b{u}(0)|^2_{\b{H^1}(T)} \\&+ \epsilon\sum_{T \in \mathcal{T}_h}|\b{u_I}(0) -\b{u_h}(0)|^2_{\b{H^1}(T)} 
				+ \dfrac{1}{\epsilon}T h^2\|\b{\dot}{\b{u}}\|^2_{\b{L^{\infty}}(\mathcal{I}; \b{H^2}(\O))}+  {\epsilon}T\norm{\b{u_I} -{\b{u_h}}}^2_{\b{L^{\infty}}(\mathcal{I};\b{V_h})}.
			\end{align*}
			Next we handle term $A_2$ in the following way
			\begin{align*}
				A_2 &= 2\int\limits_{0}^{t}  \sum_{e\in \mathcal{E}^C_h} \int\limits_{e} \b{\sigma}_{\tau}(\b{u}(s)) \cdot (\b{\dot}{\b{u_I}}(s) -\b{\dot}{\b{u_h}}(s))_{\tau}~d{\sigma} ds + 2\int\limits_{0}^{t} \big(j( \b{\dot}{\b{u_I}}(s)) - j(\bm{\dot{u_h}}(s))\big)~ds  \\
				&\leq 2 \int\limits_{0}^{t}  \sum_{e \in \mathcal{E}^C_h} \int\limits_{e} \b{\sigma}_{\tau}(\b{u}(s))\cdot (\b{\dot}{\b{u_I}}(s))_{\tau}~d\sigma ds + \int\limits_{0}^{t}  \sum_{e \in \mathcal{E}^C_h} \int\limits_{e}  g_a  |(\b{\dot}{\b{u_h}}(s))_\tau|~d\sigma ds \\&+ 2\int\limits_{0}^{t} \sum_{e\in \mathcal{E}^C_h} \int\limits_e \big(g_a  |(\b{\dot}{\b{u_I}}(s))_\tau| -g_a  |(\b{\dot}{\b{u_h}}(s))_\tau|\big)~d\sigma ds \\
				&=  2\int\limits_{0}^{t}  \sum_{e \in \mathcal{E}^C_h} \int\limits_{e} \b{\sigma}_{\tau}(\b{u}(s)) \cdot \b{\dot}({\b{u_I}}(s))_{\tau}~d\sigma ds + 2 \int\limits_{0}^{t}  \sum_{e \in \mathcal{E}^C_h} \int\limits_{e} g_a  |(\b{\dot}{\b{u_I}}(s))_\tau| ~d\sigma ds. \\
			\end{align*}
			Using the fact that for any contact edge $e$, we have $ \b{\sigma}_{\tau}({\b{u}}(t)) \cdot (\b{\dot}{\b{u}}(t))_\tau= -g_a|(\b{\dot}{\b{u}}(t))_\tau|$, we deduce 
			\begin{align*}
				A_2 &= 2\int\limits_{0}^{t}  \sum_{e \in \mathcal{E}^C_h} \int\limits_{e} \b{\sigma}_{\tau}(\b{u}(s)) \cdot (\b{\dot}{\b{u_I}}(s)-\b{\dot}{\b{u}}(s))_{\tau}~d\sigma ds + 2 \int\limits_{0}^{t}  \sum_{e \in \mathcal{E}^C_h} \int\limits_{e} g_a  |(\b{\dot}{\b{u_I}}(s))_\tau| ~d\sigma ds\\
				&-2 \int\limits_{0}^{t}  \sum_{e \in \mathcal{E}^C_h} \int\limits_{e} g_a  |(\b{\dot}{\b{u}}(s))_\tau| ~d\sigma ds \\
				&\leq 4 \int\limits_{0}^{t}  \sum_{e \in \mathcal{E}^C_h} \int\limits_{e} g_a  |(\b{\dot}{\b{u_I}}(s))_\tau - (\b{\dot}{\b{u}}(s))_\tau| ~d\sigma ds \\
				& \lesssim \|g_a\|_{{L^\infty}(\Gamma_C)} \int\limits_{0}^{t} \sum_{e \in \mathcal{E}^C_h}h_e^{\frac{1}{2}} \|\b{\dot}{\b{u_I}}(s) - \b{\dot}{\b{u}}(s)\|_{\b{L^2}(e)}~ds.
			\end{align*}
			By applying trace inequality, we find
			\begin{align*}
				A_2 & \leq 2\|g_a\|_{L^{\infty}(\Gamma_C)} h^2 |\b{\dot{u}}|_{\b{L^{\infty}}(\mathcal{I};\b{H^2}(\Omega))}.
			\end{align*}
			Now,  we will proceed to bound $A_3$. 
			\begin{align}\label{paper5_eqn67}
				A_3 &= 2\int\limits_{0}^{t} \sum_{e\in \mathcal{E}^0_h} \int\limits_{e} \smean{\b{\sigma}(\b{u}(s))}: \sjump{\b{\dot}{\b{u_I}}(s) -\b{\dot}{\b{u_h}}(s)}~d\sigma ds  \nonumber \\
				&= 2\int\limits_{0}^{t} \sum_{e\in \mathcal{E}^0_h} \int\limits_{e} \smean{\b{\sigma}(\b{u}(s)- \b{u_I}(s))}: \sjump{\b{\dot}{\b{u_I}}(s) -\b{\dot}{\b{u_h}}(s)}~d\sigma ds \nonumber \\&+ 2\int\limits_{0}^{t} \sum_{e\in \mathcal{E}^0_h} \int\limits_{e} \smean{\b{\sigma}(\b{u_I}(s))}: \sjump{\b{\dot}{\b{u_I}}(s) -\b{\dot}{\b{u_h}}(s)}~d\sigma ds.
			\end{align}
			First we focus on second term of last equality. Utilizing that $\b{\sigma}(\b{u_I}(t))$ is a constant term on each element, we have 
			\begin{align*}
				2\int\limits_{0}^{t} \sum_{e\in \mathcal{E}^0_h} \int\limits_{e} \smean{\b{\sigma}(\b{u_I})(s)}: \sjump{\b{\dot}{\b{u_I}}(s) -\b{\dot}{\b{u_h}}(s)}~d\sigma ds &=\b{0}
			\end{align*}
			since $\sjump{\b{\dot}{\b{u_I}}(t)-\b{\dot}{\b{u_h}}(t)}(m_e) =\bm{0}$. Performing integration by parts with respect to time in the first term of equality \eqref{paper5_eqn67}, we obtain
			\begin{align}\label{paper5_8}
				2\int\limits_{0}^{t} \sum_{e\in \mathcal{E}^0_h} \int\limits_{e} \smean{\b{\sigma}(\b{u}(s)- \b{u_I}(s))}: \sjump{\b{\dot}{\b{u_I}}(s) &-\b{\dot}{\b{u_h}}(s)}~d\sigma ds \nonumber\\ &= -2\int\limits_{0}^{t} \sum_{e\in \mathcal{E}^0_h} \int\limits_{e} \smean{\b{\sigma}(\b{\dot}{\b{u}}(s)- \b{\dot}{\b{u_I}}(s))}: \sjump{\b{u_I}(s) -{\b{u_h}}(s)}~d\sigma ds \nonumber\\ &+2\sum_{e\in \mathcal{E}^0_h} \int\limits_{e} \smean{\b{\sigma}(\b{u}(t)- \b{u_I}(t))}: \sjump{{\b{u_I}}(t) -{\b{u_h}}(t)}~d{\sigma} \nonumber\\
				&\lesssim  \int\limits_{0}^{t} \sum_{e\in \mathcal{E}^0_h} |\b{\dot}{\b{u}}(s)- \b{\dot}{\b{u_I}}(s)|_{\b{H^1}(e)}\|\sjump{\b{u_I}(s) -{\b{u_h}}(s)}\|_{\b{L^2}(e)}~ds  \nonumber\\
				&+2 \sum_{e\in \mathcal{E}^0_h} |\b{u}(t)- \b{u_I}(t)|_{\b{H^1}(e)}\| \sjump{{\b{u_I}(t)} -{\b{u_h}(t)}}\|_{\b{L^2}(e)}.
			\end{align}
			Employing trace inequality and properties of interpolation operator in relation \eqref{paper5_8}, we find
			\begin{align}\label{paper5_11}
				2\int\limits_{0}^{t} \sum_{e\in \mathcal{E}^0_h} \int\limits_{e} \smean{\b{\sigma}(\b{u}(s)&- \b{u_I}(s))}: \sjump{\b{\dot}{\b{u_I}}(s) -\b{\dot}{\b{u_h}}(s)}~d\sigma ds \nonumber \\& \lesssim  \int\limits_{0}^{t} \sum_{e\in \mathcal{E}^0_h} h_e^{\frac{1}{2}}\|\b{\dot}{\b{u}}(s)\|_{\b{H^2}(\mathcal{T}_e)}\|\sjump{\b{u_I}(s) -{\b{u_h}}(s)}\|_{\b{L^2}(e)}~ds \notag \\&+ \sum_{e\in \mathcal{E}^0_h} h_e^{\frac{1}{2}}\|{\b{u}(t)}\|_{\b{H^2}(\mathcal{T}_e)}\|\sjump{\b{u_I}(t)-{\b{u_h}(t)}}\|_{\b{L^2}(e)}.
			\end{align}
			As $\sjump{\b{\dot}{\b{u_I}}-\b{\dot}{\b{u_h}}}$ becomes zero at the midpoint of every edge $e \in \mathcal{E}_h$,  therefore we use the Poincar$\acute{e}$-Witinger inequality \cite{kesavan1989topics}, followed by the trace inverse inequality to achieve
			\begin{align*}
				2\int\limits_{0}^{t} \sum_{e\in \mathcal{E}^0_h} \int\limits_{e} \smean{\b{\sigma}(\b{u}(s)&- \b{u_I}(s))}: \sjump{\b{\dot}{\b{u_I}}(s) -\b{\dot}{\b{u_h}}(s)}~d\sigma ds\\ & \lesssim  \int\limits_{0}^{t} \sum_{e\in \mathcal{E}^0_h} h_e\|\b{\dot}{\b{u}}(s)\|_{\b{H^2}(\mathcal{T}_e)}|{\b{u_I}(s)-{\b{u_h}}(s)}|_{\b{H^1}(\mathcal{T}_e)} ~ds\\ &+ \sum_{e\in \mathcal{E}^0_h} h_e\|{\b{u}(t)}\|_{\b{H^2}(\mathcal{T}_e)}|\b{u_I}(t)-{\b{u_h}(t)}|_{\b{H^1}(\mathcal{T}_e)}.
			\end{align*}
			Finally, an application of Young's inequality yields
			\begin{align*}
				2\int\limits_{0}^{t} \sum_{e\in \mathcal{E}^0_h} \int\limits_{e} \smean{\b{\sigma}(\b{u}(s)&- \b{u_I}(s))}: \sjump{\b{\dot}{\b{u_I}}(s) -\b{\dot}{\b{u_h}}(s)}~d\sigma ds \\& \lesssim  \int\limits_{0}^{t} \bigg( \dfrac{1}{\epsilon}\sum_{e\in \mathcal{E}^0_h} h_e^2\|\b{\dot}{\b{u}}(s)\|^2_{\b{H^2}(\mathcal{T}_e)}~ds +\epsilon\sum_{e\in \mathcal{E}^0_h} |{\b{u_I}(s)-{\b{u_h}(s)}}|^2_{\b{H^1}(\mathcal{T}_e)} \bigg)~ds\\ &+ \dfrac{1}{\epsilon}\sum_{e\in \mathcal{E}^0_h} h_e^2\|{\b{u}(t)}\|^2_{\b{H^2}(\mathcal{T}_e)} +\epsilon\sum_{e\in \mathcal{E}^0_h} |{\b{u_I}(t)-{\b{u_h}(t)}}|^2_{\b{H^1}(\mathcal{T}_e)} \\
				&\lesssim \frac{1}{\epsilon} h^2_e\|{\b{u}}\|^2_{\b{L^{\infty}}(\mathcal{I};\b{H^2}(\Omega))} + \frac{1}{\epsilon} h^2_e\|\b{\dot}{\b{u}}\|^2_{\b{L^2}(\mathcal{I};\b{H^2}(\Omega))} \\&+ \vspace{0.5 cm} \epsilon \norm{\b{u_I} - \b{u_h}}^2_{\b{L^\infty}(\mathcal{I};\b{V_h})}.
			\end{align*}
			Therefore, by substituting the bounds obtained for $A_1$, $A_2$ and $A_3$ in relation \eqref{paper5_5} and choosing $\epsilon>0$ sufficiently small, we obtain
			\begin{align*}
				\norm{\b{u_I}(t) - \b{u_h}(t)}^2_{\b{L^\infty}(\mathcal{I};\b{V_h})} \lesssim h^2.
			\end{align*}
			Finally, combining the upper bound of $\norm{\b{\theta}(t)}_{\b{L^\infty}(\mathcal{I};\b{V_h})},$ we arrive at the final result.
		\end{proof}
		\vspace{-0.4 cm}
	\end{theorem}
	\section{Fully-Discrete Approximation}\label{ch5sec5}
	In addition to divide the spatial domain into the regular triangulation, we also partition the time interval $[0, T]$ into subintervals $[t_{\n-1}, t_\n]$ for $\n = 1, 2, \ldots, N$, where $0 = t_0 < t_1 < t_2 < t_3 \ldots < t_N = T$. We represent the length of  the subinterval $[t_{\n-1}, t_\n]$ as $k_\n$, with $k$ denoting the maximum value among all $k_\n$. At each time step $t_\n$, the fully-discrete solution is identified by $\b{u}^{\n}$. 
	For the linear functionals $\b{f} $ and $\b{g} $ we use the notation $\b{f}^{\n} = \b{f}(t_\n)$ and $~\b{g}^{\n} = \b{g}(t_\n)$ to describe their value at time level $t_\n$. It is important to note that the above notations are well-defined due to Sobolev imbedding theorem, which gaurantees that $\b{W^{1,\infty}}(\mathcal{I}; \b{X}) \hookrightarrow \b{C}(\mathcal{I}; \b{X})$ for any Banach space $\b{X}.$ We have employed backward divided difference scheme to approximate the derivatives $i.e.~\delta \b{u}^{\n}= \frac{{u}^{\n}-{u}^{\n-1}}{{k_\n}}.$
	\par 
	Using the backward Euler scheme, we obtain the following fully-discrete 
	approximation of variational inequality $\eqref{paper5_1.9}.$
	
	\textbf{Problem 2}~(\textit{Fully-Discrete Variational Inequality})
	\textit{For each} $\n=1,2,\cdot \cdot N$, \textit{find} $\big\{\b{u_h}^\n \big\}_{\n=0}^{N}\in \b{\mathcal{K}^h_{CR}}$ \textit{such that} 
	\begin{align}
		a_h(\bm{u_h}^{\n}, \b{v_h}- \delta\b{u_h}^\n) +  j(\b{v_h}) - j(\delta\b{u_h}^\n) &\geq (\bm{l^{\n}}, \b{v_h}- \delta\b{u_h}^\n) \quad \forall~\b{v_h} \in  \b{\mathcal{K}^h_{CR}}, \label{paper5_FDI} \\
		\bm{u_h}^0 &= \mathbf{I}^{\b{CR}}\bm{u_o}. \label{paper51_FDI_1}
	\end{align}
	The existence and uniqueness of the solution of Problem 2 is  guaranteed by Lions and Stampacchia Theorem \cite{ciarlet2002finite}. 
	\subsection{{Error Estimates for Fully-Discrete Approximation}}
	\par 
	\noindent
	\vspace{0.3 cm}
	\\ 
	In this subsection, we will derive \emph{a priori} error estimates for fully-discrete scheme. 
	\begin{theorem}
		Let $\b{u}$ be the solution of continuous problem \eqref{paper5_1.9}-\eqref{paper5_1.10} and $\big\{\b{u_h}^{\text{\upshape{\n}}} \big\}_{\text{\upshape{\n}}=0}^{N}$ be the solution of fully-discrete problem \eqref{paper5_FDI}-\eqref{paper51_FDI_1}. Assume $\b{u}, \b{\dot}{\b{u}}\in \b{L^{\infty}}[\mathcal{I};\b{H^2}(\Omega)]$ and $ \b{\ddot}{\b{u}}  \in \b{L^{1}}[\mathcal{I};\b{H^2}(\Omega)]$ for any $t \in \mathcal{I}$. Then, we have 
		\begin{align*}
			\underset{\text{\upshape{\n}}}{max}~\norm{\b{u}^\text{\upshape{\n}} -\b{u_h}^\text{\upshape{\n}}} \lesssim h +k.
		\end{align*}
	\end{theorem}
	\begin{proof}
		Let $\b{e}^\n = \b{u}^\n- \b{u_h}^\n$. Decompose error into $\b{e}^\n= \b{\theta}^\n +\b{\eta}^\n$ where $\b{\eta}^\n = \b{u}^\n- \b{u_I}^\n$ and $\b{\theta}^\n = \b{u_I}^\n- \b{u_h}^\n$. Taking into the account the continuity of bilinear form $a_h(\cdot, \cdot)$ with respect to the norm $\norm{\cdot}$, we first obtain the lower bound of $a_h(\b{e}^\n, \delta\b{e}^\n)$ as follows
		\begin{align*}
			a_h(\b{e}^\n, \delta\b{e}^\n) &= \dfrac{1}{k_n} \bigg(a_h(\b{e}^\n ,\b{e}^\n)- a_h(\b{e}^\n ,\b{e}^{\n-1}) \bigg)\\
			&\geq  \dfrac{1}{k_n} \bigg(\norm{\b{e}^\n}^2- \norm{\b{e}^\n}\norm{ \b{e}^{\n-1}} \bigg) \\
			&\geq  \dfrac{1}{k_n} \bigg(\norm{\b{e}^\n}^2- \norm{\b{e}^{\n-1}}^2 \bigg).
		\end{align*}
		Next, we determine the upper bound for $a_h(\b{e}^\n, \delta\b{e}^\n)$ in the following manner:
		\begin{align}\label{paper5_eq11}
			a_h(\b{e}^\n, \delta\b{e}^\n) &= a_h(\b{e}^\n, \delta\b{\theta}^\n) +a_h(\b{e}^\n, \delta\b{\eta}^\n) \nonumber \\
			&= a_h(\b{e}^\n, \delta\b{\eta}^\n) +a_h(\b{u}^\n, \delta\b{u_I}^\n- \delta\b{u_h}^\n)-a_h(\b{u_h}^\n, \delta\b{u_I}^\n- \delta\b{u_h}^\n).
		\end{align}
		Since $\delta\b{u_I}^\n\in \b{\mathcal{K}^h_{CR}}$, using equation \eqref{paper5_FDI} we find 
		\begin{align}\label{paper5_eq12}
			-a_h(\bm{u_h}^{\n}, \delta\b{u_I}^\n- \delta\b{u_h}^\n) &\leq j(\delta\b{u_I}^\n) - j(\delta\b{u_h}^\n) - (\bm{l^{\n}}, \delta\b{u_I}^\n- \delta\b{u_h}^\n).
		\end{align}
		Inserting equation $\eqref{paper5_eq12}$ in relation $\eqref{paper5_eq11}$, we get 
		\begin{align}\label{paper5_eq13}
			a_h(\b{e}^\n, \delta\b{e}^\n) &\leq a_h(\b{e}^\n, \delta\b{\eta}^\n) +a_h(\b{u}^\n, \delta\b{u_I}^\n- \delta\b{u_h}^\n) + j(\delta\b{u_I}^\n) - j(\delta\b{u_h}^\n) - (\bm{l^{\n}}, \delta\b{u_I}^\n- \delta\b{u_h}^\n).
		\end{align}
		Owing to intgeration by parts in the second term of above inequality, we get
		\begin{align}\label{paper5_eq144}
			a_h(\b{u}^\n, \delta\b{u_I}^\n- \delta\b{u_h}^\n) &= \sum _{T \in \mathcal{T}_h}\int\limits_T \b{\sigma}(\b{u}^\n): \b{\epsilon}(\delta\b{u_I}^\n- \delta\b{u_h}^\n)~dx \nonumber  \\
			&= \sum_{T\in \mathcal{T}_h} \int\limits_{T} \b{f}(t_\n)\cdot (\delta\b{u_I}^\n- \delta\b{u_h}^\n)~dx +\sum_{e\in \mathcal{E}^N_h} \int\limits_{e} \b{g}(t_\n)\cdot (\delta\b{u_I}^\n- \delta\b{u_h}^\n)~d{\sigma} \nonumber \\ &+  \sum_{e\in \mathcal{E}^C_h} \int\limits_{e} (\b{\sigma}(\b{u^\n})\b{n}) \cdot (\delta\b{u_I}^\n- \delta\b{u_h}^\n)~d{\sigma} + \sum_{e\in \mathcal{E}^0_h} \int\limits_{e} \smean{\b{\sigma}(\b{u}^\n)}: \sjump{\delta\b{u_I}^\n- \delta\b{u_h}^\n}~d{\sigma}.
		\end{align}
		Combining equations \eqref{paper5_eq13} and \eqref{paper5_eq144}, we obtain
		\begin{align}\label{paper5_eq133}
			a_h(\b{e}^\n, \delta\b{e}^\n) &\leq a_h(\b{e}^\n, \delta\b{\eta}^\n) + \sum_{e\in \mathcal{E}^0_h} \int\limits_{e} \smean{\b{\sigma}(\b{u}^\n)}: \sjump{\delta\b{u_I}^\n- \delta\b{u_h}^\n}~d{\sigma}+ j(\delta\b{u_I}^\n) - j(\delta\b{u_h}^\n) \nonumber\\
			&+  \sum_{e\in \mathcal{E}^C_h} \int\limits_{e} (\b{\sigma}(\b{u^\n})\b{n}) \cdot (\delta\b{u_I}^\n- \delta\b{u_h}^\n)~d{\sigma} \nonumber \\
			&= T_1+T_2+T_3, 
		\end{align}
		where 
		\begin{align*}
			T_1&:= \sum_{e\in \mathcal{E}^0_h} \int\limits_{e} \smean{\b{\sigma}(\b{u}^\n)}: \sjump{\delta\b{u_I}^\n- \delta\b{u_h}^\n}~d{\sigma}, \\
			T_2&:= j(\delta\b{u_I}^\n) - j(\delta\b{u_h}^\n) +  \sum_{e\in \mathcal{E}^C_h} \int\limits_{e} (\b{\sigma}(\b{u^\n})\b{n}) \cdot (\delta\b{u_I}^\n- \delta\b{u_h}^\n)~d{\sigma},  \\
			T_3&:= a_h(\b{e}^\n, \delta\b{\eta}^\n). 
		\end{align*}
		$\bullet$ \textit{Estimates for $T_1$}:
		\begin{align*}
			T_1&= \sum_{e\in \mathcal{E}^0_h} \int\limits_{e} \smean{\b{\sigma}(\b{u}^\n)}: \sjump{\delta\b{u_I}^\n- \delta\b{u_h}^\n}~d{\sigma}\\ &=  \sum_{e\in \mathcal{E}^0_h} \int\limits_{e} \smean{\b{\sigma}(\b{u}^\n -\b{u_I}^\n )}: \sjump{\delta\b{u_I}^\n- \delta\b{u_h}^\n}~d{\sigma} + \sum_{e\in \mathcal{E}^0_h} \int\limits_{e} \smean{\b{\sigma}(\b{u_I}^\n )}: \sjump{\delta\b{u_I}^\n- \delta\b{u_h}^\n}~d{\sigma}. 
		\end{align*}
		Observe that the second term in the last relation vanishes as $\smean{\b{\sigma}(\b{u_I}^\n )}$ is constant and $\sum_{e\in \mathcal{E}^0_h} \int\limits_{e}  \sjump{\delta\b{u_I}^\n- \delta\b{u_h}^\n}~d{\sigma} = 0$ on any edge $e \in \mathcal{E}^0_h.$ Thus $T_1$ reduces to
		\begin{align}\label{paper5_eq14}
			T_1&=  \sum_{e\in \mathcal{E}^0_h} \int\limits_{e} \smean{\b{\sigma}(\b{u}^\n -\b{u_I}^\n )}: \sjump{\delta\b{u_I}^\n- \delta\b{u_h}^\n}~d{\sigma} \nonumber \\&= \sum_{e\in \mathcal{E}^0_h} \int\limits_{e} \smean{\b{\sigma}(\b{u}^\n -\b{u_I}^\n )}: \sjump{\delta\b{u_I}^\n- \delta\b{u}^\n}~d{\sigma} + \sum_{e\in \mathcal{E}^0_h} \int\limits_{e} \smean{\b{\sigma}(\b{u}^\n -\b{u_I}^\n )}: \sjump{\delta\b{u}^\n- \delta\b{u_h}^\n}~d{\sigma} \nonumber \\
			&= \sum_{e\in \mathcal{E}^0_h} \int\limits_{e} \smean{\b{\sigma}(\b{u}^\n -\b{u_I}^\n )}: \sjump{\delta\b{u_I}^\n}~d{\sigma} + \sum_{e\in \mathcal{E}^0_h} \int\limits_{e} \smean{\b{\sigma}(\b{u}^\n -\b{u_I}^\n )}: \sjump{\delta\b{u}^\n- \delta\b{u_h}^\n}~d{\sigma} \nonumber \\ 
			&= \sum_{e\in \mathcal{E}^0_h} \int\limits_{e} \smean{\b{\sigma}(\b{u}^\n -\b{u_I}^\n )}: \sjump{\delta\b{u_I}^\n - \b{\dot}{\b{u}}^\n}~d{\sigma} + \sum_{e\in \mathcal{E}^0_h} \int\limits_{e} \smean{\b{\sigma}(\b{u}^\n -\b{u_I}^\n )}: \sjump{\delta\b{u}^\n- \delta\b{u_h}^\n}~d{\sigma}.
		\end{align}
		Observe that 
		\begin{align*}
			\delta\b{u_I}^\n = \dfrac{\b{u_I}^\n - \b{u_I}^{\n-1}}{k_\n}=  \dfrac{\b{u_I}^\n -( \b{u_I}^{\n}-k_\n \b{\dot{u_I}}^\n)}{k_\n} =  \b{\dot{u_I}}^\n.
		\end{align*}
		A use of trace inequality and properties of interpolation map in equation \eqref{paper5_eq14} yields
		\begin{align}\label{paper5_eq141}
			T_1
			& \leq \sum_{e\in \mathcal{E}^0_h} \int\limits_{e} \smean{\b{\sigma}(\b{u}^\n -\b{u_I}^\n )}: \sjump{\delta\b{u}^\n- \delta\b{u_h}^\n}~d{\sigma} + \sum_{e\in \mathcal{E}^0_h}|\b{u}^\n -\b{u_I}^\n|_{\b{H^1}(e)}\|\b{\dot{u_I}}^\n - \b{\dot}{\b{u}}^\n\|_{\b{L^2}(e)} \\
			& \leq \sum_{e\in \mathcal{E}^0_h} \int\limits_{e} \smean{\b{\sigma}(\b{u}^\n -\b{u_I}^\n )}: \sjump{\delta\b{u}^\n- \delta\b{u_h}^\n}~d{\sigma} + h^2\|\b{u}^\n\|_{\b{H^2}(\Omega)}\|\b{\dot}{\b{u}}^\n\|_{\b{H^2}(\Omega)} \nonumber \\
			& \leq \sum_{e\in \mathcal{E}^0_h} \int\limits_{e} \smean{\b{\sigma}(\b{u}^\n -\b{u_I}^\n )}: \sjump{\delta\b{u}^\n- \delta\b{u_h}^\n}~d{\sigma} + 2h^2\|\b{u}^\n\|^2_{\b{H^2}(\Omega)}+2h^2\|\b{\dot}{\b{u}}^\n\|^2_{\b{H^2}(\Omega)}.
		\end{align}
		Using Pioncar$\acute{e}$-Witinger inequality followed by inverse trace inequality, we find $$\|\sjump{\delta\b{u_I}^\n- \delta\b{u_h}^\n}\|_{\b{L^2(e)}} \leq h_e^{\frac{1}{2}}|\delta\b{u_I}^\n- \delta\b{u_h}^\n|_{\b{H^1}(\mathcal{T}_e)}~\forall~e\in \mathcal{E}^0_h.$$Thus, we have 
		\begin{align}\label{paper5_eq1411}
			T_1
			& \leq \sum_{e\in \mathcal{E}^0_h} \int\limits_{e} \smean{\b{\sigma}(\b{u}^\n -\b{u_I}^\n )}: \sjump{\delta\b{u}^\n- \delta\b{u_h}^\n}~d{\sigma} + \sum_{e\in \mathcal{E}^0_h}h_e|\b{u}^\n |_{\b{H^2}(\cT_e)}|\delta\b{u_I}^\n|_{\b{H^1}(\mathcal{T}_e)}.
		\end{align}
		Also, since , we obtain 
		\begin{align}\label{paper5_eq141115}
			T_1
			& \leq \sum_{e\in \mathcal{E}^0_h} \int\limits_{e} \smean{\b{\sigma}(\b{u}^\n -\b{u_I}^\n )}: \sjump{\delta\b{u}^\n- \delta\b{u_h}^\n}~d{\sigma} + h|\b{u}^\n |_{\b{H^2}(\Omega)}|\b{\dot}{\b{u}}^\n|_{\b{H^1}(\Omega)}.
		\end{align}
		$\bullet$ \textit{Estimates for $T_2$}:
		\begin{align}
			T_2&= j(\delta\b{u_I}^\n) - j(\delta\b{u_h}^\n) +  \sum_{e\in \mathcal{E}^C_h} \int\limits_{e} (\b{\sigma}(\b{u^\n})\b{n}) \cdot (\delta\b{u_I}^\n- \delta\b{u_h}^\n)~d{\sigma} \nonumber \\
			&=  \sum_{e\in \mathcal{E}^C_h} \int\limits_{e} (\b{\sigma}_{\tau}(\b{u^\n})) \cdot (\delta\b{u_I}^\n- \delta\b{u_h}^\n)_\tau~d{\sigma} +  j(\delta\b{u_I}^\n) - j(\delta\b{u_h}^\n)  \nonumber \\ 
			&=  \sum_{e\in \mathcal{E}^C_h} \int\limits_{e} (\b{\sigma}_{\tau}(\b{u^\n})) \cdot (\delta\b{u_I}^\n)_\tau~d{\sigma} +  \sum_{e\in \mathcal{E}^C_h} \int\limits_{e} g_a|(\delta\b{u_h}^\n)_\tau |~d{\sigma} \nonumber \\
			&+ \sum_{e\in \mathcal{E}^C_h} \int\limits_{e} g_a|(\delta\b{u_I}^\n)_\tau |~d{\sigma}- \sum_{e\in \mathcal{E}^C_h} \int\limits_{e} g_a|(\delta\b{u_h}^\n)_\tau |~d{\sigma} \nonumber \\
			&=  \sum_{e\in \mathcal{E}^C_h} \int\limits_{e} (\b{\sigma}_{\tau}(\b{u^\n})) \cdot (\delta\b{u_I}^\n)_\tau~d{\sigma} +  \sum_{e\in \mathcal{E}^C_h} \int\limits_{e} g_a|(\delta\b{u_I}^\n)_\tau |~d{\sigma}.
		\end{align}
		Since $\delta\b{u_I}^\n =\b{\dot{u_I}}^\n$, we have 
		\begin{align*}
			T_2&= \sum_{e\in \mathcal{E}^C_h} \int\limits_{e} (\b{\sigma}_{\tau}(\b{u^\n})) \cdot (\b{\dot{u_I}}^\n - \b{\dot{u}}^\n)_\tau~d{\sigma} + \sum_{e\in \mathcal{E}^C_h} \int\limits_{e} (\b{\sigma}_{\tau}(\b{u^\n})) \cdot ( \b{\dot{u}}^\n)_\tau~d{\sigma} + \sum_{e\in \mathcal{E}^C_h} \int\limits_{e} g_a|(\delta\b{u_I}^\n)_\tau |~d{\sigma}
		\end{align*}
		Exploiting the fact that $\b{\sigma}_{\tau}(\b{u^\n}) \cdot  \b{\dot{u}}^\n_\tau = -g_a |\b{\dot{u}}^\n_\tau|$, the last relation reduces to 
		\begin{align*}
			T_2 &\leq  2\sum_{e\in \mathcal{E}^C_h}\int\limits_eg_a |\b{\dot{u_I}}^\n - \b{\dot{u}}^\n_\tau|~d{\sigma} \leq 2\|g_a\|_{L^{\infty}(\Gamma_C)}\sum_{e\in \mathcal{E}^C_h}\int\limits_e|\b{\dot{u_I}}^\n - \b{\dot{u}}^\n_\tau|~d{\sigma}
		\end{align*}
		Using Cauchy-Schwarz inequality followed by trace inequality and Lemma \ref{Interpolation1}, we arrive at 
		\begin{align*}
			T_2 & \leq 2\|g_a\|_{L^{\infty}(\Gamma_C)}\sum_{e\in \mathcal{E}^C_h}h_e^{\frac{1}{2}}\|\b{\dot{u_I}}^\n - \b{\dot{u}}^\n_\tau\|_{\b{L^2}(e)} \\
			& \leq 2\|g_a\|_{L^{\infty}(\Gamma_C)}\sum_{e\in \mathcal{E}^C_h} \sum_{T \in \mathcal{T}_e} \bigg( \|\b{\dot{u_I}}^\n - \b{\dot{u}}^\n_\tau\|_{\b{L^2}(T)} + h_e |\b{\dot{u_I}}^\n - \b{\dot{u}}^\n_\tau|_{\b{H^1}(T)} \bigg) \\
			&\leq  2\|g_a\|_{L^{\infty}(\Gamma_C)} h^2 |\b{\dot{u}}^\n|_{\b{H^2}(\Omega)}.
		\end{align*}
		$\bullet$ \textit{Estimates for $T_3$}:
		\begin{align*}
			T_3&= a_h(\b{e}^\n, \delta\b{\eta}^\n) = \dfrac{1}{k_n}a_h(\b{e}^\n, \b{\eta}^\n - \b{\eta}^{\n-1}) \leq \dfrac{1}{k_n} \norm{\b{e}^\n}\norm{\b{\eta}^\n - \b{\eta}^{\n-1}}.
		\end{align*}
		It suffices to bound $\norm{\b{\eta}^\n - \b{\eta}^{\n-1}}$ as follows
		\begin{align*}
			\norm{\b{\eta}^\n - \b{\eta}^{\n-1}} &= \norm{\b{u}^\n- \b{u_I}^\n- (\b{u}^{\n-1}- \b{u_I}^{\n-1})} \\
			&= \norm{\b{u}^\n- \b{u}^{\n-1}-( \b{u_I}^\n- \b{u_I}^{\n-1})}. 
		\end{align*}
		Using Taylor's theorem, we obtain
		\begin{align}\label{paper5_eqn20}
			\norm{\b{\eta}^\n - \b{\eta}^{\n-1}} &= \norm{k_n\b{\dot{u_I}}^{\n-1} - k_n\b{\dot{u}}^{\n-1} - \int\limits_{t_{\n-1}}^{t_\n}(t_{\n}-s)\b{\ddot}{\b{u}}(s)~ds}
			\nonumber \\
			&\leq \norm{k_\n\b{\dot{u_I}}^{\n-1} - k_n\b{\dot{u}}^{\n-1}}+\norm{~\int\limits_{t_{\n-1}}^{t_\n}(t_{\n}-s)\b{\ddot}{\b{u}}(s)~ds} \nonumber \\
			& \leq hk_\n |\b{\dot{u}}^{\n-1}|_{\b{H^2}(\Omega)} +k_\n \norm{~\int\limits_{t_{\n-1}}^{t_\n}\b{\ddot}{\b{u}}(s)~ds}.
		\end{align}
		Combining the estimates for $T_1$, $T_2$ and $T_3$, we obtain
		\begin{align*}
			a_h(\b{e}^\n, \delta\b{e}^\n) &\leq  \sum_{e\in \mathcal{E}^0_h} \int\limits_{e} \smean{\b{\sigma}(\b{u}^\n -\b{u_I}^\n )}: \sjump{\delta\b{u}^\n- \delta\b{u_h}^\n}~d{\sigma} + 2\|g_a\|_{L^{\infty}(\Gamma_C)} h^2 |\b{\dot{u}}^\n|_{\b{H^2}(\Omega)} \\
			&+\frac{1}{k_n}\bigg(\norm{ \b{e}^\n} \big( hk_\n |\b{\dot{u}}^{\n-1}|_{\b{H^2}(\Omega)} +k_\n ~\int\limits_{t_{\n-1}}^{t_\n}\norm{\b{\ddot}{\b{u}}(t)}~dt\big) \bigg).
		\end{align*}
		Exploiting the lower bound of $a_h(\b{e}^\n, \delta\b{e}^\n)$, we deduce
		\begin{align*}
			\dfrac{1}{k_n} \bigg(\norm{\b{e}^\n}^2- \norm{\b{e}^{\n-1}}^2 \bigg) &\leq \sum_{e\in \mathcal{E}^0_h} \int\limits_{e} \smean{\b{\sigma}(\b{u}^\n -\b{u_I}^\n )}: \sjump{\delta\b{u}^\n- \delta\b{u_h}^\n}~d{\sigma} + 2\|g_a\|_{L^{\infty}(\Gamma_C)} h^2 |\b{\dot{u}}^\n|_{\b{H^2}(\Omega)} \\&
			+\frac{1}{k_n}\bigg(\norm{ \b{e}^\n} \big( hk_\n |\b{\dot{u}}^{\n-1}|_{\b{H^2}(\Omega)} +k_\n ~\int\limits_{t_{\n-1}}^{t_\n}\norm{\b{\ddot}{\b{u}}(t)}~dt\big) \bigg).
		\end{align*}
		This implies
		\begin{align*}
			\norm{\b{e}^\n}^2- \norm{\b{e}^{\n-1}}^2&\leq \sum_{e\in \mathcal{E}^0_h} \int\limits_{e} k_n\smean{\b{\sigma}(\b{u}^\n -\b{u_I}^\n )}: \sjump{\delta\b{u}^\n- \delta\b{u_h}^\n}~d{\sigma} + k_j (h^2\|\b{u}^\n\|^2_{\b{H^2}(\Omega)}+h^2\|\b{\dot}{\b{u}}^\n\|^2_{\b{H^2}(\Omega)})\\&+ 2\|g_a\|_{L^{\infty}(\Gamma_C)} h^2k_{\n} |\b{\dot{u}}^\n|_{\b{H^2}(\Omega)} 
			+ hk_\n \norm{ \b{e}^\n} |\b{\dot{u}}^{\n-1}|_{\b{H^2}(\Omega)} +k_\n \norm{ \b{e}^\n} ~\int\limits_{t_{\n-1}}^{t_\n}\norm{\b{\ddot}{\b{u}}(t)}~dt.
		\end{align*}
		A simple induction from $\n=1,2,.....,N$ yields
		\begin{align}\label{ch5_equation18}
			\norm{\b{e}^N}^2&\leq \sum_{j=1}^N \sum_{e\in \mathcal{E}^0_h} \int\limits_{e} k_j \smean{\b{\sigma}(\b{u}^j -\b{u_I}^j )}: \sjump{\delta\b{u}^j- \delta\b{u_h}^j}~d{\sigma}+  (h^2\underset{\n}{\max}\|\b{u}^\n\|^2_{\b{H^2}(\Omega)} \nonumber \\&+h^2\underset{\n}{\max}\|\b{\dot}{\b{u}}^\n\|^2_{\b{H^2}(\Omega)})\sum_{j=1}^N k_j+ 2\|g_a\|_{L^{\infty}(\Gamma_C)} h^2\sum_{j=1}^N k_j ~\underset{\n}{\max}~|\b{\dot{u}}^\n|_{\b{H^2}(\Omega)} 
			\nonumber \\&+ h ~ \underset{\n}{\max}~\norm{ \b{e}^\n}~\underset{\n}{\max}~ |\b{\dot{u}}^{\n-1}|_{\b{H^2}(\Omega)} \sum_{j=1}^N k_j
			+~\underset{\n}{\max}~\norm{ \b{e}^\n} \sum_{j=1}^N k_j  ~\int\limits_{t_{j-1}}^{t_j}\norm{\b{\ddot}{\b{u}}(t)}~dt.
		\end{align}
		Next, we will estimate the first term in the above relation as follows
		\begin{align*}
			\sum_{j=1}^N \sum_{e\in \mathcal{E}^0_h} \int\limits_{e} k_j \smean{\b{\sigma}(\b{u}^j -\b{u_I}^j )}: \sjump{\delta\b{u}^j&- \delta\b{u_h}^j}~d{\sigma}\\&=	\sum_{j=1}^N \sum_{e\in \mathcal{E}^0_h} \int\limits_{e}  \smean{\b{\sigma}(\b{u}^j -\b{u_I}^j )}: \sjump{(\b{u}^j - \b{u}^{j-1}) - (\b{u_h}^j-\b{u_h}^{j-1})}~d{\sigma} \\
			&=\sum_{j=1}^N \sum_{e\in \mathcal{E}^0_h} \int\limits_{e} \sjump{(\b{u}^j - \b{u_h}^j ) - (\b{u}^{j-1}-\b{u_h}^{j-1})} : \smean{\b{\sigma}(\b{u}^j -\b{u_I}^j )} ~d{\sigma}.
		\end{align*}	
		Upon re-arranging these terms we obtain
		\begin{align*}
			\sum_{j=1}^N \sum_{e\in \mathcal{E}^0_h} \int\limits_{e} \sjump{(\b{u}^j &- \b{u_h}^j ) - (\b{u}^{j-1}-\b{u_h}^{j-1})} : \smean{\b{\sigma}(\b{u}^j -\b{u_I}^j )} ~d{\sigma} \notag\\&= \sum_{e\in \mathcal{E}^0_h} \int\limits_{e}\sjump{\b{u}^N - \b{u_h}^N} : \smean{\b{\sigma}(\b{u}^N -\b{u_I}^N )} ~d{\sigma} -  \sum_{e\in \mathcal{E}^0_h} \int\limits_{e}\sjump{\b{u}^0 - \b{u_h}^0} : \smean{\b{\sigma}(\b{u}^1 -\b{u_I}^1 )} ~d{\sigma} \notag\\&+ 	\sum_{j=1}^{N-1} \sum_{e\in \mathcal{E}^0_h} \int\limits_{e} \sjump{\b{u}^j - \b{u_h}^j } : \smean{\b{\sigma}(\b{u}^j -\b{u_I}^j )-\b{\sigma}(\b{u}^{j+1} -\b{u_I}^{j+1} )} ~d{\sigma}. 
		\end{align*}	
		Using Cauchy-Schwarz inequality, we find 
		\begin{align*}
			\sum_{j=1}^N \sum_{e\in \mathcal{E}^0_h} \int\limits_{e} \sjump{(\b{u}^j &- \b{u_h}^j ) - (\b{u}^{j-1}-\b{u_h}^{j-1})} : \smean{\b{\sigma}(\b{u}^j -\b{u_I}^j )} ~d{\sigma} \notag \\ \nonumber&= \sum_{e\in \mathcal{E}^0_h} \|\sjump{\b{u}^N - \b{u_h}^N}\|_{\b{L^2}(e)}\|\smean{\b{\sigma}(\b{u}^N -\b{u_I}^N )}\|_{\b{L^2}(e)} \\&+ \sum_{e\in \mathcal{E}^0_h} \|\sjump{\b{u}^0 - \b{u_h}^0}\|_{\b{L^2}(e)}\|\smean{\b{\sigma}(\b{u}^1 -\b{u_I}^1 )}\|_{\b{L^2}(e)}\\  &+ 	\sum_{j=1}^{N-1} \sum_{e\in \mathcal{E}^0_h} \|\sjump{\b{u}^j - \b{u_h}^j }\|_{\b{L^2}(e)} \|\smean{\b{\sigma}(\b{u}^j -\b{u_I}^j )-\b{\sigma}(\b{u}^{j+1} -\b{u_I}^{j+1} )}\|_{\b{L^2}(e)} \\
			&\leq \norm{\b{e}^N}h_e|\b{u}^N|_{\b{H^2}(\Omega)} + \norm{\b{e}^0}h_e|\b{u}^1|_{\b{H^2}(\Omega)}\\  &+ 	\sum_{j=1}^{N-1} \sum_{e\in \mathcal{E}^0_h} \|\sjump{\b{u}^j - \b{u_h}^j }\|_{\b{L^2}(e)} \|\smean{\b{\sigma}(\b{u}^j -\b{u_I}^j )-\b{\sigma}(\b{u}^{j+1} -\b{u_I}^{j+1} )}\|_{\b{L^2}(e)}
		\end{align*}
		as a result of trace inequality and properties of interpolation map. Now, consider
		\begin{align}\label{ch5_eq25}
			\sum_{j=1}^{N-1} \sum_{e\in \mathcal{E}^0_h} &\|\sjump{\b{u}^j - \b{u_h}^j }\|_{\b{L^2}(e)} \|\smean{\b{\sigma}(\b{u}^j -\b{u_I}^j )-\b{\sigma}(\b{u}^{j+1} -\b{u_I}^{j+1} )}\|_{\b{L^2}(e)} \nonumber\\&\leq \sum_{j=1}^{N-1} \bigg( \sum_{e\in \mathcal{E}^0_h} h_e^{-1}\|\sjump{\b{u}^j - \b{u_h}^j }\|^2_{\b{L^2}(e)} \bigg)^{\frac{1}{2}} \bigg( \sum_{e\in \mathcal{E}^0_h} h_e\|\smean{\b{\sigma}(\b{u}^j -\b{u_I}^j )-\b{\sigma}(\b{u}^{j+1} -\b{u_I}^{j+1} )}\|^2_{\b{L^2}(e)}\bigg)^{\frac{1}{2}}
			\nonumber	\\&\leq \sum_{j=1}^{N-1} \norm{\b{e}^j}\bigg( \sum_{e\in \mathcal{E}^0_h} h_e\|\smean{\b{\sigma}(\b{u}^j -\b{u_I}^j )-\b{\sigma}(\b{u}^{j+1} -\b{u_I}^{j+1} )}\|^2_{\b{L^2}(e)}\bigg)^{\frac{1}{2}}
			\nonumber	\\&= \sum_{j=1}^{N-1} \norm{\b{e}^j}\bigg( \sum_{e\in \mathcal{E}^0_h} h_e|(\b{u}^j -\b{u_I}^j )-(\b{u}^{j+1} -\b{u_I}^{j+1} )|^2_{\b{H^1}(e)}\bigg)^{\frac{1}{2}} \nonumber\\
			&= \sum_{j=1}^{N-1} \norm{\b{e}^j}\bigg( \sum_{e\in \mathcal{E}^0_h} h_e k_{j+1}^2|\delta\b{\eta}^{j+1}|^2_{\b{H^1}(e)}\bigg)^{\frac{1}{2}}.
		\end{align}
		For any $j=1,2,\hdots, N-1 $,  using trace inequality we find
		\begin{align}\label{ch5_sign23}
			|\delta\b{\eta}^{j+1}|_{\b{H^1}(e)} &\leq h_e^{-\frac{1}{2}}|\delta\b{\eta}^{j+1}|_{\b{H^1}(\cT_e)} +h_e^{\frac{1}{2}} |\delta\b{\eta}^{j+1}|_{\b{H^2}(\cT_e)} \nonumber \\ 
			&= h_e^{-\frac{1}{2}}|\delta\b{\eta}^{j+1}|_{\b{H^1}(\cT_e)} +h_e^{\frac{1}{2}} |\delta\b{u}^{j+1}|_{\b{H^2}(\cT_e)} 
		\end{align}
		since $\b{u_I}^j$ is a linear polynomial. An application of Taylor's theorem yields
		\begin{align*}
			\delta\b{u}^{j+1} = \frac{1}{k_{j+1}}\bigg( k_{j+1} \b{\dot}{\b{u}}^j + \int_{t_j}^{t_{j+1}} (t_{j+1}-t)\b{\ddot}{\b{u}}(t)~dt \bigg).
		\end{align*}
		Thus, 
		\begin{align}\label{ch5_sign20}
			|\delta\b{u}^{j+1}|_{\b{H^2}(T)}\leq | \b{\dot}{\b{u}}^j |_{\b{H^2}(T)} + \int_{t_j}^{t_{j+1}}| \b{\ddot}{\b{u}}(t)|_{\b{H^2}(T)}~ dt.
		\end{align}
		Also, using relation derived in equation \eqref{paper5_eqn20} we find
		\begin{align}\label{ch5_sign21}
			|\delta\b{\eta}^{j+1}|_{\b{H^1}(T)} &= \frac{1}{k_{j+1}}|\b{\eta}^{j+1}-\b{\eta}^{j}|_{\b{H^1}(T)} \leq h_e| \b{\dot}{\b{u}}^j |_{\b{H^2}(T)} + \int_{t_j}^{t_{j+1}}| \b{\ddot}{\b{u}}(t)|_{\b{H^1}(T)}~ dt.
		\end{align}
		Combining the relations obtained in equations \eqref{ch5_sign23}, \eqref{ch5_sign20} and \eqref{ch5_sign21}, we obtain
		\begin{align}
			|\delta\b{\eta}^{j+1}|_{\b{H^1}(e)} &\leq h_e^{\frac{1}{2}}\| \b{\dot}{\b{u}}^j \|_{\b{H^2}(T)} + h_e^{-\frac{1}{2}} \int_{t_j}^{t_{j+1}}| \b{\ddot}{\b{u}}(t)|_{\b{H^1}(T)}~ dt + h_e^{\frac{1}{2}} \int_{t_j}^{t_{j+1}}| \b{\ddot}{\b{u}}(t)|_{\b{H^2}(T)}~ dt.
		\end{align}
		Thus, the equation \eqref{ch5_eq25} stems down to 
		\begin{align}\label{ch5_eq251}
			\sum_{j=1}^{N-1} &\sum_{e\in \mathcal{E}^0_h} \|\sjump{\b{u}^j - \b{u_h}^j }\|_{\b{L^2}(e)} \|\smean{\b{\sigma}(\b{u}^j -\b{u_I}^j )-\b{\sigma}(\b{u}^{j+1} -\b{u_I}^{j+1} )}\|_{\b{L^2}(e)} \nonumber\\& \leq \sum_{j=1}^{N-1} \norm{\b{e}^j}\bigg( h_e^2k_{j+1}^2\| \b{\dot}{\b{u}}^j \|^2_{\b{H^2}(T)} + k_{j+1}^2 \big(\int_{t_j}^{t_{j+1}}| \b{\ddot}{\b{u}}(t)|_{\b{H^1}(T)}~ dt \big)^2+ h_e^{2} k_{j+1}^2 \big( \int_{t_j}^{t_{j+1}}| \b{\ddot}{\b{u}}(t)|_{\b{H^2}(T)}~ dt\big)^2 \bigg)^{\frac{1}{2}} \nonumber \\
			& \leq \sum_{j=1}^{N-1} \norm{\b{e}^j}\Bigg( h_ek_{j+1}\| \b{\dot}{\b{u}}^j \|_{\b{H^2}(T)} + k_{j+1} \int_{t_j}^{t_{j+1}}| \b{\ddot}{\b{u}}(t)|_{\b{H^1}(T)}~ dt + h_e k_{j+1} \int_{t_j}^{t_{j+1}}| \b{\ddot}{\b{u}}(t)|_{\b{H^2}(T)}~ dt\Bigg). \nonumber
		\end{align}
		Thus, the relation \eqref{ch5_equation18} reduces to
		\begin{align}
			\norm{\b{e}^N}^2&\leq h_e~\underset{\n}{\max}~|\b{\dot{u}}^\n|_{\b{H^2}(\Omega)} \underset{\n}{\max}~\norm{ \b{e}^\n} 	\sum_{j=1}^N k_{j} +  h_ek~\underset{\n}{\max}~\norm{ \b{e}^\n}  \bigg(\int_{\mathcal{I}}| \b{\ddot}{\b{u}}(t)|_{\b{H^1}(\Omega)}~ dt \bigg)\nonumber \\ &+ k~ \underset{\n}{\max}~\norm{ \b{e}^\n} \int_{\mathcal{I}}| \b{\ddot}{\b{u}}(t)|_{\b{H^2}(\Omega)}~ dt 
			+ 2\|g_a\|_{L^{\infty}(\Gamma_C)} h^2\sum_{j=1}^N k_j ~\underset{\n}{\max}~|\b{\dot{u}}^\n|_{\b{H^2}(\Omega)} 
			\\ \nonumber &+ h ~\underset{\n}{\max}~\norm{ \b{e}^\n}~\underset{\n}{\max}~ |\b{\dot{u}}^{\n-1}|_{\b{H^2}(\Omega)} 
			+~\underset{\n}{\max}~\norm{ \b{e}^\n} \sum_{j=1}^N k_j  ~\int\limits_{t_{j-1}}^{t_j}\norm{\b{\ddot}{\b{u}}(t)}~dt \\
			&+ h ~\underset{\n}{\max}~\norm{ \b{e}^\n}~\underset{\n}{\max}~ |\b{{u}}^{\n}|_{\b{H^2}(\Omega)} + b(h^2\underset{\n}{\max}\|\b{u}^\n\|^2_{\b{H^2}(\Omega)}+h^2\underset{\n}{\max}\|\b{\dot}{\b{u}}^\n\|^2_{\b{H^2}(\Omega)}).
		\end{align}	
		Let $M= \underset{\n}{max}~\norm{\b{e}^\n}$. Then we have
		\begin{align*}
			M^2&\leq M \big(h~|\b{u}^\n|_{\b{L^\infty}({\mathcal{I};\b{H^2}(\Omega))}}+h~|\b{\dot}{\b{u}}^\n|_{\b{L^\infty}({\mathcal{I};\b{H^2}(\Omega))}}+h^2\norm{\b{\ddot{u}}}_{\b{L^1}(\mathcal{I};\b{V})} + k^2\norm{\b{\ddot{u}}}_{\b{L^1}(\mathcal{I};\b{V})}  \\&+ k\norm{\b{\ddot{u}}}_{\b{L^1}(\mathcal{I};\b{H^2}(\Omega))}\big) + h^2|\b{u}|^2_{\b{L^\infty}({\mathcal{I};\b{H^2}(\Omega))}} + k^2\norm{\b{\ddot{u}}}^2_{\b{L^1}(\mathcal{I};\b{V})}+ h^2|\b{\dot}{\b{u}}|^2_{\b{L^\infty}({\mathcal{I};\b{H^2}(\Omega))}}.
		\end{align*}
		Using the identity if $x^2 \leq ax +b$ where $x,a,b>0 $ then $x \leq a+b^{\frac{1}{2}}$, we obtain
		\begin{align*}
			\underset{\n}{max}~\norm{\b{e}^\n} \lesssim h+k.
		\end{align*}
		This completes the proof.
	\end{proof}
	\section{Numerical Experiments}\label{ch5sec6}
	In this section, we present the numerical experiment to examine the convergence characteristics of the fully discrete scheme on uniform triangular grids. The experiment has been carried out in MATLAB$\_$R2020B. In order to transform the variational inequality \eqref{paper5_FDI} into an equality, we introduce Lagrange multiplier $\b{\lambda_h} \in \b{L^\infty}(\Gamma_C)$ using the Uzawa algorithm \cite{glowinski2008lectures}. With this the variational inequality \eqref{paper5_FDI} is equivalent to 
	\begin{align*}
		a_h(\bm{u_h}^{\n}, \b{v_h}) +  \int\limits_{\Gamma_C} g_a \b{\lambda_{h\tau}}\cdot \b{v_{h\tau}} = (\bm{l^{\n}}, \b{v_h}) \quad \forall~\b{v_h} \in  \b{\mathcal{K}^h_{CR}}.
	\end{align*}
	The following is the concise overview of the Uzawa algorithm implemented at each time step $t_\n$.
	\par
	\noindent
	\vspace{0.5 cm}
	\textbf{Step 1~:} Initialise $\b{\lambda^0_h}=\b{0}$ and compute the solution $[\b{u^\n_h}]^0 \in \b{\mathcal{K}^h_{CR}}$ of the discrete variational equality.
	\begin{align*}
		a_h([\bm{u_h}^{\n}]^0, \b{v_h})  = (\bm{l^{\n}}, \b{v_h}) \quad \forall~\b{v_h} \in  \b{\mathcal{K}^h_{CR}}.
	\end{align*}
	\vspace{0.3 cm}
	\textbf{Step 2~:} Define a projection map $P$ as
	\begin{align*}
		P(\chi) := \text{sup}(-1,  \text{inf}
		(1,\chi))~~~\forall~\chi\in\b{L^{\infty}}(\Gamma_C). 
	\end{align*} 
	With the help of this projection operator, update the Lagrange multiplier $\b{\lambda}^k_{\b{h}}$,~$k=1,2,...$ as 
	\begin{align*}
		\b{\lambda}^k_{\b{h}} = P(\b{\lambda}^{k-1}_{\b{h}}+\tilde{\rho} c_{\tau}{{[\b{u}^{\n}_{\b{h}}]^{k-1}}})
	\end{align*}
	and solve for $[\b{u}^{\n}_{\b{h}}]^k$ from the following equation
	\begin{align*}
		a_h([\b{u}^{\n}_{\b{h}}]^k, \b{v_h})=&(\bm{l^{\n}}, \b{v_h})+\int\limits_{\Gamma_C}c_\tau \b{\lambda}^k_{\b{h\tau}}\cdot \b{v_{h\tau}}~ds \quad \forall~\b{v_h} \in  \b{\mathcal{K}^h_{CR}}
	\end{align*}
	where, $\tilde{\rho}$ is a positive constant and is chosen to be 1 in the next example.
	\par
	\vspace{0.3 cm}
	\noindent
	\textbf{Step 3 :} The iteration process is terminated when $|[\b{u}^{\n}_{\b{h}}]^k -[\b{u}^{\n}_{\b{h}}]^{k-1}| < \epsilon$ where $\epsilon$ denotes the error tolerance.
	\par
	Now we present the numerical outcome of a two-dimensional example. Therein, we employ fully discrete scheme and approximate the derivatives by backward Euler difference. The domain under consideration is a square domain and the traingles are refined uniformly at each iteration. We set the error tolerance $\epsilon$ as $10^{-8}$ and the penalty parameter $\rho$ is chosen as 10.  
	\begin{figure*}[!ht]			
		\centering
		\includegraphics[width=0.55\textwidth]{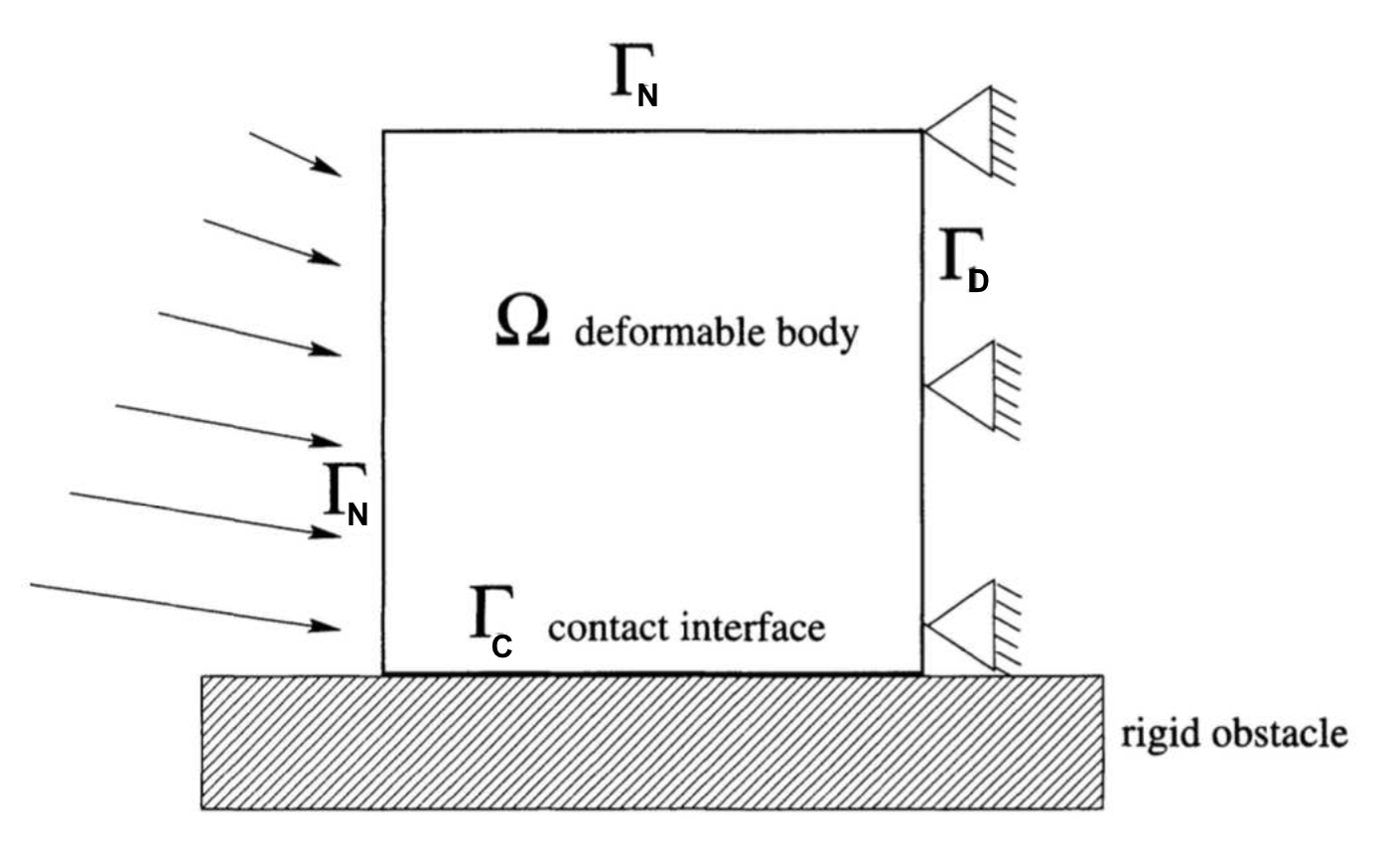}
		\caption{Physical Setting for Example 5.1.}
		\label{ch5_0}
	\end{figure*}
	\begin{example}\label{exp1}
		In this model problem, we consider deformation of the domain $\O:= (0,4)\times(0,4)$ which represents the cross-section of a three dimensional linearly elastic body subjected to certain traction forces. The body is clamped on $\Gamma_D:= \{4\} \times (0,4)$ and therefore the displacement field vanishes there. Surface traction forces act on the part of Neumann boundary $\{0\} \times (0,4)$ whereas, the part $(0,4) \times \{4\}$ is traction free. Thus, $\Gamma_N:= \big(\{0\} \times (0,4) \big) \cup \big((0,4) \times \{4\}\big).$ The body is in contact with the rigid foundation on the part of contact boundary $\Gamma_C:=(0,4) \times \{0\}.$ 
		We consider discretizing the space $\b{V}$ using linear elements and dividing the time interval $I$ into uniform partitions. For our calculations, we have employed the following data, where the unit $daN/mm^2$ corresponds to deca-newtons per square millimeter.
		\begin{align*}
			E &= 200,~~ \nu=0.3, ~~\b{f}=(0,0)~daN/mm^2,\\
			\b{g}&=(0.02(5-y)t,-0.01t)~daN/mm^2,\\
			g_a&=0.0012~daN/mm^2,~~\b{u}^0 = \b{0}~m,~~T=1 ~sec.
		\end{align*}
	\end{example}
	\par 
	\vspace{-0.4 cm}
	As the exact solution $\b{u}$ is unknown in the above example, we calculate the energy norm error by calculating the difference between the discrete solutions  obtained on the consecutive meshes. The pictorial representation of the above example is depicted in Figure \ref{ch5_0}. Table \ref{ch5_Table1} and Table \ref{ch5_Table2} depicts the energy norm error and its order of convergence for the spatial variable $h$ and time variable $k$. Therein, $\textbf{DOF}$ refers to the degrees of freedom and $N$ denotes the number of time steps. We observe that the numerical convergence rates are near to 1, which are in accordance with the theoretical results derived in Section \ref{ch5sec5}.
	
	\begin{table}
		\centering
		\begin{tabular}{|c|c|c|c|c|c|} 
			\hline
			$N$&${\b{h}}$&$\b{k}$&$\textbf{DOF}$ & $\textbf{Total~error}$ & \textbf{Order} \\
			\hline
			$40$ &$2^{-1}$ & $0.025$ &$28$ &2.512 $\times$ $10^{-4}$&- \\ 
			$80$&$2^{-2}$ &  0.0125         &$104$&1.431 $\times$ $10^{-4}$  & 0.8579\\ 
			$160$&$2^{-3}$ &  0.00625      &$400$&7.460  $\times$ $10^{-5}$ & 0.9657\\ 
			$320$&$2^{-4}$ & 0.003125      &$1568$ &4.030 $\times$ $10^{-5}$  &      0.9006\\
			$640$&$2^{-5}$ & 0.0015625        &$6208$ &2.2210 $\times$ $10^{-5}$ &  0.8939\\
			\hline
		\end{tabular}
		\caption{Error and order of convergence with respect to mesh parameter $h$ for Example 5.1.}
		\label{ch5_Table1}
	\end{table}
	\begin{table}
		\centering
		\begin{tabular}{|c|c|c|c|c|c|} 
			\hline
			$N$&${\b{h}}$ & $\b{k}$&$\textbf{DOF}$ & $\textbf{Total~error}$ & \textbf{Order} \\
			\hline
			$20$ &$2^{-2}$ &  0.05         & $104$ &1.411 $\times$ $10^{-4}$&- \\ 
			$40$&$2^{-3}$ &   0.025        & $400$&7.410 $\times$ $10^{-4}$  & 0.9292\\ 
			$80$&$2^{-4}$ & 0.0125          &$1568$&4.010  $\times$ $10^{-5}$ & 0.8859\\ 
			$160$&$2^{-5}$ &  0.00625         &$6208$ &2.1220  $\times$ $10^{-5}$  &   0.9182\\
			$320$&$2^{-6}$ & 0.003125         &$12416$ &1.0810 $\times$ $10^{-5}$ &  0.9731\\
			\hline
		\end{tabular}
		\caption{Error and order of convergence with respect to time parameter $k$ for Example 5.1.}
		\label{ch5_Table2}
	\end{table}

	\bibliographystyle{unsrt}
	\bibliography{Bibliography}
\end{document}